\documentclass[11pt]{article}
\usepackage{amsmath}
\usepackage{dsfont}
\usepackage{mathrsfs}
\usepackage{amsmath,amssymb}
\usepackage{amsfonts}
\usepackage{hyperref}
\usepackage{amsthm}
\usepackage{graphicx}
\usepackage{subfigure}
\usepackage{xcolor}

\hfuzz=\maxdimen
\newtheorem{lemma}{Lemma}[section]
\newtheorem{definition}[lemma]{Definition}
\newtheorem{proposition}[lemma]{Proposition}
\newtheorem{theorem}[lemma]{Theorem}
\newtheorem{corollary}[lemma]{Corollary}
\newtheorem{example}{Example}
\newtheorem{remark}{Remark}
\newtheorem{question}{Question}

\newtheorem{conjecture}{Conjecture}
\numberwithin{equation}{section}

\DeclareFixedFont{\Acknowledgment}{OT1}{cmr}{bx}{n}{14pt}
\begin{document}

\title{\begin{bf} On the deformation of ball packings\end{bf}}
\author{Huabin Ge, Wenshuai Jiang, Liangming Shen}
\maketitle
\begin{abstract}
In this paper, we study the geometric aspects of ball packings on $(M,\mathcal{T})$, where $\mathcal{T}$ is a triangulation on a 3-manifold $M$. We introduce a combinatorial Yamabe invariant $Y_{\mathcal{T}}$, depending on the topology of $M$ and the combinatoric of $\mathcal{T}$. We prove that $Y_{\mathcal{T}}$ is attainable if and only if there is a constant curvature packing, and the combinatorial Yamabe problem can be solved by minimizing Cooper-Rivin-Glickenstein functional. We then study the combinatorial Yamabe flow introduced by Glickenstein \cite{G0}-\cite{G2}. We first prove a small energy convergence theorem which says that the flow would converge to a constant curvature metric if the initial energy is close in a quantitative way to the energy of a constant curvature metric. We shall also prove: although the flow may develop singularities in finite time, there is a natural way to extend the solution of the flow so as it exists for all time. Moreover, if the triangulation $\mathcal{T}$ is regular (that is, the number of tetrahedrons surrounding each vertex are all equal), then the combinatorial Yamabe flow converges exponentially fast to a constant curvature packing.
\end{abstract}
\tableofcontents

\section{Introduction}
\subsection{Background and Preliminaries}
Circle packing builds a connection between combinatoric and geometry. It was used by Thurston \cite{T1} to construct hyperbolic 3-manifolds or 3-orbifolds. Inspired by Thurston's work, Cooper and Rivin \cite{CR} studied the deformation of ball packings, which are the three dimensional analogues of circle packings. Glickenstein \cite{G1}\cite{G2} then introduced a combinatorial version of Yamabe flow based on Euclidean triangulations coming form ball packings. Besides their works, there is very little known about the deformation of ball packings. In this paper, we shall study the deformation of ball packings.

Let $M$ be a closed 3-manifold with a triangulation $\mathcal{T}=\{\mathcal{T}_0,\mathcal{T}_1,\mathcal{T}_2,\mathcal{T}_3\}$, where the symbols $\mathcal{T}_0,\mathcal{T}_1,\mathcal{T}_2,\mathcal{T}_3$ represent the sets of vertices, edges, faces and tetrahedrons respectively. In the whole paper, we denote by $\{ij\}$, $\{ijk\}$ and $\{ijkl\}$ a particular edge, triangle and tetrahedron respectively in the triangulation, while we denote by $\{i,j,\cdots\}$ a particular set with elements $i$, $j, \cdots$. The symbol $(M, \mathcal{T})$ will be referred as a triangulated manifold in the following. All the vertices are ordered one by one, marked by $1,\cdots,N$, where $N=|\mathcal{T}_0|$ is the number of vertices. We use $i\thicksim j$ to denote that the vertices $i$, $j$ are adjacent if there is an edge $\{ij\}\in\mathcal{T}_1$ with $i$, $j$ as end points.
For a triangulated three manifold $(M,\mathcal{T})$, Cooper and Rivin \cite{CR} constructed a piecewise linear metric by ball packings. A ball packing (also called sphere packing in other literatures) is a map $r:\mathcal{T}_0\rightarrow (0,+\infty)$ such that the length between vertices $i$ and $j$ is $l_{ij}=r_{i}+r_{j}$ for each edge $\{ij\}\in \mathcal{T}_1$, and each combinatorial tetrahedron $\{ijkl\}\in\mathcal{T}_3$ with six edge lengthes $l_{ij},l_{ik},l_{il},l_{jk},l_{jl},l_{kl}$ forms an Euclidean tetrahedron. Geometrically, a ball packing $r=(r_1,\cdots,r_N)$ attaches to each vertex $i\in\mathcal{T}_0$ a ball $S_i$ with $i$ as center and $r_i$ as radius, and if $i\thicksim j$, the two balls $S_i$ and $S_j$ are externally tangent. Cooper and Rivin \cite{CR} called the Euclidean tetrahedrons generated in this way \emph{conformal} and proved that an Euclidean tetrahedron is conformal if and only if there exists a unique ball tangent to all of the edges of the tetrahedron. Moreover, the point of tangency with the edge $\{ij\}$ is of distance $r_i$ to the vertex $i$. Denote
\begin{equation}\label{nondegeneracy condition}
Q_{ijkl}=\left(\frac{1}{r_{i}}+\frac{1}{r_{j}}+\frac{1}{r_{k}}+\frac{1}{r_{l}}\right)^2-
2\left(\frac{1}{r_{i}^2}+\frac{1}{r_{j}^2}+\frac{1}{r_{k}^2}+\frac{1}{r_{l}^2}\right).
\end{equation}
The classical Descartes' circle theorem, also called Soddy-Gossett theorem (for example, see \cite{CR}), says that four circles in the plance of radii $r_i$, $r_j$, $r_k$, $r_l$ are externally tangent if and only if $Q_{ijkl}=0$. This case is often called Apollonian circle packings. It is surprising that Apollonian circle packings are closely related to number theory \cite{Bourgain}\cite{Bour-Kon} and hyperbolic geometry \cite{K-Oh}. Glickenstein \cite{G1} pointed out that a combinatorial tetrahedron $\{ijkl\}\in \mathcal{T}_3$ configured by four externally tangent balls with positive radii $r_{i}$, $r_{j}$, $r_{k}$ and $r_{l}$ can be realized as an Euclidean tetrahedron if and only if $Q_{ijkl}>0$. Denote $\mathcal{M}_{\mathcal{T}}$ by the space of all ball packings, then it can be expressed as a subspace of $\mathds{R}^N_{>0}$:
\begin{equation}
 \mathcal{M}_{\mathcal{T}}=\left\{\;r\in\mathds{R}^N_{>0}\;\big|\;Q_{ijkl}>0, \;\forall \{ijkl\}\in \mathcal{T}_3\;\right\}.
\end{equation}
Cooper and Rivin proved that $ \mathcal{M}_{\mathcal{T}}$ is a simply connected open subset of $\mathds{R}^N_{>0}$. It is a cone, but not convex. For each ball packing $r$, there corresponds a combinatorial scalar curvature $K_{i}$ at each vertex $i\in\mathcal{T}_0$, which was introduced by Cooper and Rivin \cite{CR}. Denote $\alpha_{ijkl}$ as the solid angle at the vertex $i$ in the Euclidean tetrahedron $\{ijkl\}\in \mathcal{T}_3$, then the combinatorial scalar curvature at $i$ is defined as
\begin{equation}\label{Def-CR curvature}
K_{i}= 4\pi-\sum_{\{ijkl\}\in \mathcal{T}_3}\alpha_{ijkl},
\end{equation}
where the sum is taken over all tetrahedrons in $\mathcal{T}_3$ with $i$ as one of its vertices. They also studied the deformation of ball packings and proved a locally rigidity result about the combinatorial curvature $K=(K_1,\cdots,K_N)$, namely, a conformal tetrahedron cannot be deformed while keeping the solid angles fixed. Or say, the combinatorial curvature map
$$K:\mathcal{M}_{\mathcal{T}}\to \mathds{R}^N,$$
up to scaling, is locally injective. Recently, Xu \cite{Xu} showed that $K$ is injective globally.

Given the triangulation $\mathcal{T}$, Cooper and Rivin consider $\mathcal{M}_{\mathcal{T}}$ as an analogy of the smooth conformal class $\{e^fg:f\in C^{\infty}(M)\}$ of a smooth Riemannian metric $g$ on $M$. Inspired by this observation, Glickenstein \cite{G1} posed the following combinatorial Yamabe problem:\\[8pt]
\noindent
\textbf{Combinatorial Yamabe Problem:} \emph{Is there a ball packing with constant combinatorial scalar curvature in the combinatorial conformal class $\mathcal{M}_{\mathcal{T}}$? How to find it?}\\[8pt]
To approach this problem, Glickensteinp introduced a combinatorial Yamabe flow
\begin{equation}\label{Def-Flow-Glickenstein}
\frac{dr_i}{dt}=-K_ir_i,
\end{equation}
aiming to deform the ball packings to one with constant (or prescribed) scalar curvature. The prototype of (\ref{Def-Flow-Glickenstein}) is Chow and Luo's combinatorial Ricci flow \cite{CL1} and Luo's combinatorial Yamabe flow \cite{L1} on surfaces. Following Chow, Luo and Glickenstein's pioneering work, the first author of this paper and his collaborates Jiang, Xu, Zhang, Ma, Zhou also introduced and studied several combinatorial curvature flows in \cite{Ge}-\cite{GZX}.

We must emphasize that combinatorial curvature flows are quite different with their smooth counterparts. It is well known that the solutions to smooth normalized Yamabe (Ricci, Calabi) flows on surfaces exit for all time $t\geq0$ and converges to metrics with constant curvature. However, as numerical simulations indicate, some combinatorial versions of surface Yamabe (Ricci, Calabi) flows may collapse in finite time. Worse still, even one may extend the flows to go through collapsing, the extended flows may not converge and may develop singularities at infinity. We suggest the readers to see Luo's combinatorial Yamabe flow \cite{L1}\cite{GJ1}, and Chow-Luo's combinatorial Ricci flows with inversive distance circle packings \cite{CL1}\cite{GJ2}-\cite{GJ4}. Since all definitions of combinatorial curvature flows are based on triangulations, singularities occur exactly when some triangles or tetrahedrons collapse. Since the triangulation deeply influences the behavior of the combinatorial flows, it takes much more effort to deal with combinatorial curvature flows.

In view of the trouble it caused in the study of surface combinatorial curvature flows, one can imagine the difficulties one will meet in the study of 3d-combinatorial curvature flows. Very little is known until now about those flows. In this paper, we will show the global convergence of the (extended) combinatorial Yamabe flow for regular triangulations. As far as we know, this is the first global convergence result for 3d-combinatorial Yamabe flows.



\subsection{Main results}
\label{section-main-result}
Cooper and Rivin first pintroduced the ``total curvature functional" $\mathcal{S}=\sum_{i=1}^N K_i r_i$. Glickenstein then considered the following ``average scalar curvature" functional
\begin{equation}
\lambda(r)=\frac{\sum_{i=1}^NK_ir_i}{\sum_{i=1}^Nr_i}, \;r\in\mathcal{M}_{\mathcal{T}},
\end{equation}
we call which the ``Cooper-Rivin-Glickenstein functional" in the paper. In the following, we abbreviate it as \emph{CRG-functional}. Let the \emph{combinatorial Yamabe invariant} $Y_{\mathcal{T}}$ be defined as (see Definition \ref{Def-alpha-normalize-Regge-functional})
\begin{equation}
Y_{\mathcal{T}}=\inf_{r\in \mathcal{M}_{\mathcal{T}}} \lambda(r).
\end{equation}
We call $Y_{\mathcal{T}}$ is attainable if the CRG-functional $\lambda(r)$ has a minimum in $\mathcal{M}_{\mathcal{T}}$. Our first result is a combination of Theorem \ref{Thm-Q-min-iff-exist-const-curv-metric} and Theorem \ref{corollary-Q-2}:
\begin{theorem}\label{Thm-main-1-yamabe-invarint}
The following are all mutually equivalent.
\begin{enumerate}
  \item The Combinatorial Yamabe Problem is solvable;
  \item $Y_{\mathcal{T}}$ is attainable in $ \mathcal{M}_{\mathcal{T}}$, i.e. the CRG-functional $\lambda$ has a global minimum in $ \mathcal{M}_{\mathcal{T}}$;
  \item The CRG-functional $\lambda$ has a local minimum in $ \mathcal{M}_{\mathcal{T}}$;
  \item The CRG-functionpal $\lambda$ has a critical point in $ \mathcal{M}_{\mathcal{T}}$.
\end{enumerate}
\end{theorem}

One has an explicit formula for the gradient and the second derivative of $\mathcal{S}(r)$. This is helpful from the practical point of view, because it allows one to use more powerful algorithms to minimize $\mathcal{S}(r)$ under the constraint condition $\sum_{i\in V}r_i=1$ and thus solve the Combinatorial Yamabe Problem. An alternative way to approach the Combinatorial Yamabe Problem is to consider the following combinatorial Yamabe flow
\begin{equation}\label{Def-norm-Yamabe-Flow}
\frac{dr_i}{dt}=(\lambda-K_i)r_i,
\end{equation}
which is a normalization of Glickenstein's combinatorial Yamabe flow (\ref{Def-Flow-Glickenstein}). We will show that (see Proposition \ref{prop-converg-imply-const-exist}) if $r(t)$, the unique solution to (\ref{Def-norm-Yamabe-Flow}), exists for all time $t\geq0$ and converges to a ball packing $\hat{r}\in\mathcal{M}_{\mathcal{T}}$, then $\hat{r}$ has constant curvature. Hence (\ref{Def-norm-Yamabe-Flow}) provides a natural way to get the solution of the Combinatorial Yamabe Problem. In practice, one may design algorithms to solve the flow equation (\ref{Def-norm-Yamabe-Flow}) and further the Combinatorial Yamabe Problem, see \cite{GeMa} for example.

The solution $\{r(t)\}_{0\leq t<T}$ to (\ref{Def-norm-Yamabe-Flow}) may collapse in finite time, where $0<T\leq\infty$ is the maximal existence time of $r(t)$. Here ``collapse" means $T$ is finite. Because $\lambda-K_i$ have no definition outside $\mathcal{M}_{\mathcal{T}}$, the collapsing happens exactly when $r(t)$ touches the boundary of $\mathcal{M}_{\mathcal{T}}$. More precisely, there exists a sequence of times $t_n\rightarrow T$, and a conformal tetrahedron $\{ijkl\}$ so that $Q_{ijkl}(r(t_n))\rightarrow0$. Geometrically, the geometric tetrahedron $\{ijkl\}$ collapses. That is, the six edges of $\{ijkl\}$ with lengths $l_{ij},l_{ik},l_{il},l_{jk},l_{jl},l_{kl}$ could no more form the edges of any Euclidean tetrahedron as $t_n\rightarrow T$. To prevent finite time collapsing, we introduce a topological-combinatorial invariant (see Definition \ref{Def-chi-invariant})
\begin{equation}
\chi(\hat{r},\mathcal{T})=\inf\limits_{\gamma\in{\mathbb{S}}^{N-1};\|\gamma\|_{l^1}=0\;}\sup\limits_{0\leq t< a_\gamma}\lambda(\hat{r}+t\gamma).
\end{equation}
where $\hat{r}$ is a constant curvature ball packing, and $a_{\gamma}$ is the least upper bound of $t$ such that $\hat{r}+t\gamma\in \mathcal{M}_{\mathcal{T}}$ for all $0\le t<a_\gamma$. By controlling the initial CRG-functional $\lambda(r(0))$, we can prevent finite time collapsing. In fact, we have the following result:

\begin{theorem}
\label{thm-intro-small-energ-converg}
Assume there exists a constant curvature ball packing $\hat{r}$ and
\begin{equation}
\lambda(r(0))\leq\chi(\hat{r},\mathcal{T}).
\end{equation}
Then $r(t)$ exists on $[0,\infty)$ and converges exponentially fast to a constant curvature packing.
\end{theorem}

It's remarkable that the constant curvature packing $\hat{r}$ is unique up to scaling, since the curvature map $K$ is global injective up to scaling by \cite{CR}\cite{Xu} (that is, if $K(r)=K(r')$, then $r'=cr$ for some $c>0$). Hence in the above theorem, $r(t)$ should converges to some packing $c\hat{r}$, which is a scaling of $\hat{r}$. Since $\|r(t)\|_{l^1}=\sum_ir_i(t)$ is invariant along (\ref{Def-norm-Yamabe-Flow}), $c$ can be determined by $c\|\hat{r}\|_{l^1}=\|r(0)\|_{l^1}$.

The above theorem is essentially a ``small energy convergence" result. It is based on the fact that there exists a constant curvature packing $\hat{r}$, and moreover, $\hat{r}$ is stable. Inspired by the extension idea introduced by Bobenko, Pinkall and Springborn \cite{Bobenko}, systematically developed by Luo \cite{L2}, Luo and Yang \cite{LuoYang}, and then widely used by Ge and Jiang \cite{GJ1}-\cite{GJ4}, Ge and Xu \cite{ge-xu} and Xu \cite{Xu}, we provide an extension way to handle finite time collapsing. Given four balls $S_1$, $S_2$, $S_3$ and $S_4$ with radii $r_1$, $r_2$, $r_3$ and $r_4$. Let $l_{ij}=r_i+r_j$ be the length of the edge $\{ij\}$, $i,j\in\{1,2,3,4\}$. In case $Q_{1234}>0$, which means that the six edges of $\{1234\}$ with lengths $l_{12},l_{13},l_{14},l_{23},l_{24},l_{34}$ form the edges of an Euclidean tetrahedron, denote $\tilde{\alpha}_{ijkl}$ by the real solid angle at the vertex $i\in\{1,2,3,4\}$. In case $Q_{1234}\leq0$, those $l_{12},l_{13},l_{14},l_{23},l_{24},l_{34}$ can not form the edge lengths of any Euclidean tetrahedron. Denote $\tilde{\alpha}_{ijkl}=2\pi$ if the ball $S_i$ go through the gap between the three mutually tangent balls $S_j$, $S_k$ and $S_l$, where $\{i,j,k,l\}=\{1,2,3,4\}$, while denote $\tilde{\alpha}_{ijkl}=0$ otherwise. The construction shows that $\tilde{\alpha}_{ijkl}$ is an extension of the solid angle $\alpha_{ijkl}$, which is defined only for those $r_1$, $r_2$, $r_3$ and $r_4$ so that $Q_{1234}>0$. We call $\tilde{\alpha}_{ijkl}$ the extended solid angle. It is defined on $\mathds{R}^4_{>0}$ and is continuous by Lemma \ref{lemma-xu-extension} in Section \ref{section-extend-solid-angle}. Using the extended solid angle $\tilde{\alpha}$, we naturally get $\widetilde{K}$, a continuously extension of the curvature $K$, by
\begin{equation}\label{Def-extend-curvature}
\widetilde{K}_{i}= 4\pi-\sum_{\{ijkl\}\in \mathcal{T}_3}\tilde{\alpha}_{ijkl}.
\end{equation}
As a consequence, the CRG-functional $\lambda$ extends naturally to $\tilde{\lambda}=\sum_i \widetilde{K}_ir_i/\sum_ir_i$, which is called the \emph{extended CRG-functional}.

\begin{theorem}\label{Thm-main-2-extend-norm-Yamabe-flow}
We can extend the combinatorial Yamabe flow (\ref{Def-norm-Yamabe-Flow}) to the following
\begin{equation}
\label{Def-introduct-extend-norm-Yamabe-flow}
\frac{dr_i}{dt}=(\tilde{\lambda}-\widetilde{K}_i)r_i
\end{equation}
so that any solution to the above extended flow (\ref{Def-introduct-extend-norm-Yamabe-flow}) exists for all time $t\geq 0$.
\end{theorem}

In the following, we call every $r\in\mathcal{M}_\mathcal{T}$ a \emph{real ball packing}, and call every $r\in\mathds{R}^N_{>0}\setminus\mathcal{M}_\mathcal{T}$ a \emph{virtual ball packing}. If we mention a ball packing in this paper, we always mean a real ball packing. It can be shown (see Theorem \ref{thm-extend-flow-converg-imply-exist-const-curv-packing}) that if $\{r(t)\}_{t\geq0}$, a solution to (\ref{Def-introduct-extend-norm-Yamabe-flow}), converges to some $r_{\infty}\in\mathds{R}^N_{>0}$, then $r_{\infty}$ has constant (extended) curvature, either real or virtual. Conversely, if we assume the triangulation $\mathcal{T}$ is \emph{regular} (also called vertex transitive), i.e. a triangulation such that the same number of tetrahedrons meet at every vertex, one can deform any packing to a real one with constant curvature along the extended flow (\ref{Def-introduct-extend-norm-Yamabe-flow}).

\begin{theorem}\label{Thm-main-3-converg-to-const-curv}
Assume $\mathcal{T}$ is regular. Then the solution $r(t)$ to the extended flow (\ref{Def-introduct-extend-norm-Yamabe-flow}) converges exponentially fast to a real packing with constant curvature as $t$ goes to $+\infty$.
\end{theorem}

Generally, we can use an extended topological-combinatorial invariant $\tilde{\chi}(\hat{r},\mathcal{T})$ to control (without assuming $\mathcal{T}$ regular) $\tilde{\lambda}(r(0))$ and further control the behavior of the extended flow (\ref{Def-introduct-extend-norm-Yamabe-flow}) so that it converges to a real packing $\hat{r}$ with constant curvature. See Theorem \ref{Thm-tuta-xi-invariant-imply-converg} for more details.
Inspired by the the proof of Theorem \ref{Thm-main-3-converg-to-const-curv}, the following conjecture seems to be true.
\begin{conjecture}
Let $d_i$ be the vertex degree at $i$. Assume $|d_i-d_j|\leq 10$ for each $i,j\in V$. Then there exists a real or virtual ball packing with constant curvature.
\end{conjecture}

However, we can prove the following theorem, which builds a deep connection between the combinatoric of $\mathcal{T}$ and the geometry of $M$.
\begin{theorem}
If each vertex degree is no more than $11$, there exists a real or virtual ball packing with constant curvature.
\end{theorem}

In a future paper \cite{Ge-hua}, the techniques of this paper will be extended to hyperbolic ball packings, the geometry of which is somewhat different from Euclidean case. If all vertex degrees are no more than $22$, there are no hyperbolic ball packings, real or virtual, with zero curvature. However, there always exists a ball packing (may be virtual) with zero curvature if all degrees are no less than $23$. It is amazing that these combinatoric-geometric results can be obtained by studying a hyperbolic version of the combinatorial Yamabe flow (\ref{Def-Flow-Glickenstein}). At the last of this section, we raise a question as follows, which is our ultimate aim:
\begin{question}
Characterize the image set $K(\mathcal{M}_\mathcal{T})$ of the curvature map $K$, and $\widetilde{K}(\mathds{R}^N_{>0})$ of the extended curvature map $\widetilde{K}$ using the combinatorics of $\mathcal{T}$ and the topology of $M$.
\end{question}

For compact surface with circle packings, Thurston \cite{T1} completely solved the above question. He characterized the image set of $K$ by a class of combinatorial and topological inequalities. However, we don't know how to approach it in three dimension.

The paper is organized as follows. In Section 2, we introduce some energy functionals and the combinatorial Yamabe invariant. In Section 3, we first prove a ``small energy convergence" result. We then introduce a topological-combinatorial invariant to control the convergence behavior of the normalized flow. In Section 4, we extend the definition of solid angles and then combinatorial curvatures. Using the extended energy functional, we solve the Combinatorial Yamabe Problem. In Section 5, we study the extended flow. We shall prove the long-term existence of solutions, and prove any solution converge to a constant curvature real packing for regular triangulations. In the Appendix, we give the Schl\"{a}ffli formula and a proof of a Lemma used in the paper.

\section{Combinatorial functionals and invariants}
\subsection{Regge's Einstein-Hilbert functional}
Given a smooth Riemannian manifold $(M, g)$, let $R$ be the smooth scalar curvature, then the smooth Einstein-Hilbert functional $\mathcal{E}(g)$ is defined by
$$\mathcal{E}(g)=\int_MR d\mu_g.$$
It had been extensively studied since its relation to general relativity, the Yamabe problem and geometric curvature flows. In order to quantize gravity, Regge \cite{Re} first suggested to consider a discretization of the smooth Einstein-Hilbert functional, which is called Regge's Einstein-Hilbert functional by Champion, Glickenstein and Young \cite{CGY}. Regge's Einstein-Hilbert functional is usually called Einstein-Hilbert action in Regge calculus by physicist. It is closely related to the gravity theory for simplicial geometry. See \cite{AMM,MM,Miler} for more description.

We look back Regge's formulation briefly. For a compact 3-dimensional manifold $M^3$  with a triangulation $\mathcal{T}$, a piecewise flat metric is a map $l:E\rightarrow (0,+\infty)$ such that for every tetrahedron $\tau=\{ijkl\}\in \mathcal{T}_3$,
the tetrahedron $\tau$ with edges lengths $l_{ij},l_{ik},l_{il},l_{jk},l_{jl},l_{kl}$ can be realized as a geometric tetrahedron in Euclidean space. For any Euclidean tetrahedron $\{ijkl\}\in \mathcal{T}_3$, the dihedral angle at edge $\{ij\}$ is denoted by $\beta_{ij,kl}$. If an edge is in the interior of the triangulation, the discrete Ricci curvature at this edge is $2\pi$ minus the sum of dihedral angles at the edge. More specifically, denote $R_{ij}$ as the discrete Ricci curvature at an edge $\{ij\}\in \mathcal{T}_1$, then
\begin{equation}
R_{ij}=2\pi-\sum_{\{ijkl\}\in\mathcal{T}_3}\beta_{ij,kl},
\end{equation}
where the sum is taken over all tetrahedrons with $\{ij\}$ as one of its edges. If this edge is on the boundary of the triangulation, then the curvature should be $R_{ij}=\pi-\sum_{\{i,j,k,l\}\in T}\beta_{ij,kl}$. Using the discrete Ricci curvature, Regge's Einstein-Hilbert functional can be expressed as
\begin{equation}
\mathcal{E}(l)=\sum_{i\thicksim j}R_{ij}l_{ij},
\end{equation}
where the sum is taken over all edges $\{ij\}\in \mathcal{T}_1$. It is noticeable that $\{l^2\}$, the space of all admissible piecewise flat metrics parameterized by $l_{ij}^2$, is a nonempty connected open convex cone. This was proved by the first author Ge of the paper, Mei and Zhou \cite{GMZ} and Schrader \cite{Sch} independently. In a forthcoming paper \cite{GJ5}, we will use this observation to prove that the space of all perpendicular ball packings (any two intersecting balls intersect perpendicularly) is the whole space $\mathds{R}^N_{>0}$. Compare that the space of tangential ball packings (i.e. the ball packings considered in this paper) $\mathcal{M}_{\mathcal{T}}$ is non-convex. In the following, we will see that non-convex of the set $\mathcal{M}_{\mathcal{T}}$ is the main difficulty to make a local result global.

\subsection{Cooper and Rivin's ``total scalar curvature" functional}
\label{section-cooper-rivin-func}
For Euclidean triangulation coming from ball packings, Cooper and Rivin \cite{CR} introduced and carefully studied the following ``total scalar curvature" functional
\begin{equation}
\label{def-cooper-rivin-funct}
\mathcal{S}(r)=\sum_{i=1}^N K_i r_i, \;r\in \mathcal{M}_{\mathcal{T}}.
\end{equation}
Using this functional, they proved that the combinatorial conformal structure cannot be deformed (except by scaling) while keep the solid angles fixed, or equivalently, the set of conformal structures with prescribed solid angles are discrete. They further used this result to prove that the geometry of ball packing of the ball $\mathbb{S}^3$ whose nerve is a triangulation $\mathcal{T}$ is rigid up to M\"obius transformations. The following result is our observation.
\begin{proposition}
Given a triangulated manifold $(M^3, \mathcal{T})$, each real ball packing $r\in \mathcal{M}_{\mathcal{T}}$ induces a piecewise flat metric $l$ with $l_{ij}=r_i+r_j$. Moreover, $\mathcal{E}(l)=\mathcal{S}(r)$.
\end{proposition}
\begin{proof}
Glickenstein \cite{G4} observed that for each vertex $i\in V$,
$$K_i=\sum_{j:j\thicksim i}R_{ij},$$
which can be proved by using Euler's characteristic formula for balls. Then it follows
$$\sum_{i=1}^NK_ir_i=\sum_{i=1}^N\sum_{j:j\thicksim i}R_{ij}r_i=\sum_{j\thicksim i}R_{ij}(r_i+r_j)=\sum_{j\thicksim i}R_{ij}l_{ij}.$$
\end{proof}

\begin{lemma}\label{Lemma-Lambda-semi-positive}
(\cite{CR, Ri, G2})
Given a triangulated manifold $(M^3, \mathcal{T})$. For Cooper and Rivin's functional $\mathcal{S}=\sum K_ir_i$, the classical Schl\"{a}ffli formula says that $d\mathcal{S}=\sum K_idr_i$. This implies $\partial_{r_i}\mathcal{S}=K_i$, or
\begin{equation}
\nabla_r\mathcal{S}=K,
\end{equation}
in collum vector form. If we denote
\begin{equation}
\Lambda=Hess_r\mathcal{S}=
\frac{\partial(K_{1},\cdots,K_{N})}{\partial(r_{1},\cdots,r_{N})},
\end{equation}
then $\Lambda$ is positive semi-definite with rank $N-1$ and
the kernel of $\Lambda$ is the linear space spanned by the vector $r$.
\end{lemma}

Glickenstein \cite{G1} calculated the entries of matrix $\Lambda$ in detail, and found a new dual structure for conformal tetrahedrons. It's also his insight to elaborate $\Lambda$ as a type of combinatorial Laplace operator and to derive a discrete maximum principle for his curvature flow (\ref{Def-Flow-Glickenstein}). For each $x\in\mathds{R}^N$, denote $\|x\|_{l^1}=\sum_{i=1}^N|x_i|$. The following lemma was essentially stated by Cooper and Rivin in their pioneering study of ball packings \cite{CR}. Since the proof is elementary and is irrelevant with this paper, we postpone it to the appendix.
\begin{lemma}\label{Lemma-Lambda-positive}
(\cite{CR})
Cooper and Rivin's functional $\mathcal{S}$ is strictly convex on $ \mathcal{M}_{\mathcal{T}}\cap \{r\in\mathds{R}^N_{>0}:\|r\|_{l^1}=1\}$.
If we restrict $\mathcal{S}$ to the hyperplane $\{x\in\mathds{R}^N:\|x\|_{l^1}=1\}$, its Hessian $\Lambda'$ is strictly positive definite (the concrete meaning of $\Lambda'$ can be seen in Appendix \ref{appendix-2}).
\end{lemma}

\subsection{Glickenstein's ``average scalar curvature" functional}
If $\hat{r}\in \mathcal{M}_{\mathcal{T}}$ is a real ball packing with constant curvature, then the combinatorial scalar curvature $K_i(\hat{r})$ at each vertex $i\in V$ equals to a constant $\mathcal{S}(\hat{r})/\|\hat{r}\|_{l^1}$. Inspired by this, Glickenstein suggested to consider the following ``average scalar curvature"
\begin{equation}\label{Def-3d-Yamabe-functional}
\lambda(r)=\frac{\mathcal{S}}{\|r\|_{l^1}}=\frac{\sum_{i=1}^NK_ir_i}{\sum_{i=1}^Nr_i}, \ \ r\in \mathcal{M}_{\mathcal{T}}.
\end{equation}
Note this functional is called the Cooper-Rivin-Glickenstein functional and abbreviated as the CRG-functional in Section \ref{section-main-result}.

For a Riemannian manifold $(M, g)$, the smooth average scalar curvature is $\frac{\int_M R d\mu_g}{\int_M d\mu_g}$. For a triangulated manifold $(M^3,\mathcal{T})$, we consider $r_i$ as a volume element, which is a combinatorial analogue of $d\mu_g$. In this sense, $\|r\|_{l^1}$ is an appropriate combinatorial volume of a ball packing $r$, and Glickenstein's ``average scalar curvature" functional (\ref{Def-3d-Yamabe-functional}) is an appropriate combinatorial analogue of the smooth average scalar curvature $\frac{\int_M R d\mu_g}{\int_M d\mu_g}$. Note that the functional $\lambda(r)=\mathcal{S}/\|r\|_{l^1}$ is also a normalization of Cooper and Rivin's functional $\mathcal{S}$, and to some extent, looks like a combinatorial version of smooth normalized Einstein-Hilbert functional $\frac{\int_M R d\mu_g}{(\int_M d\mu_g)^\frac{1}{3}}$.

It is remarkable that the CRG-functional (\ref{Def-3d-Yamabe-functional}) can be generalized to $\alpha$ order. In fact, the first author of this paper Ge and Xu \cite{GX3} once defined the $\alpha$-functional
\begin{equation}\label{Def-3d-alpha-Yamabe-functional}
\lambda_{\alpha}(r)=\frac{\mathcal{S}}{\|r\|_{\alpha+1}}=\frac{\sum_{i=1}^NK_ir_i}{\big(\sum_{i=1}^Nr_i^{\alpha+1}\big)^{\frac{1}{\alpha+1}}}, \ \ r\in \mathcal{M}_{\mathcal{T}}.
\end{equation}
for each $\alpha\in \mathds{R}$ with $\alpha\neq-1$. Then the CRG-functional $\lambda(r)$ is in fact the $0$-functional defined above. The major aim for introducing the $\alpha$-functional is to study the $\alpha$-curvature $R_{\alpha,i}=K_i/r_i^{\alpha}$. The critical points of the $\alpha$-functional are exactly the ball packings with constant $\alpha$-curvature. The first author of this paper Ge and his collaborators explained carefully the motivation to study the $\alpha$-curvature, and particularly the $\alpha=2$ case, see \cite{GeMa}\cite{GX1}-\cite{GX4}. We will follow up the deformation of ball packings towards the constant (or prescribed) $\alpha$-curvatures in the subsequent studies.

\subsection{The combinatorial Yamabe invariant}
Inspired by the above analogy analysis, we introduce some combinatorial invariants here.

\begin{definition}\label{Def-alpha-normalize-Regge-functional}
The combinatorial Yamabe invariant with respect to $\mathcal{T}$ is defined as
\begin{equation}\label{def-Y-T}
Y_{\mathcal{T}}=\inf_{r\in \mathcal{M}_{\mathcal{T}}} \lambda(r),
\end{equation}
where the combinatorial Yamabe constant of $M$ is defined as $Y_{M}=\sup\limits_{\mathcal{T}}\inf\limits_{r\in \mathcal{M}_{\mathcal{T}}} \lambda(r).$
\end{definition}

For a fixed triangulation $\mathcal{T}$, all $K_i$ are uniformly bounded by the topology of $M$ and the combinatorics of $\mathcal{T}$ by the definition (\ref{Def-CR curvature}). Note $|\lambda(r)|\leq\|K\|_{l^\infty}$ for every ball packing $r\in\mathcal{M}_{\mathcal{T}}$. Hence $Y_{\mathcal{T}}$ is well defined and is a finite number. It depends on the triangulation $\mathcal{T}$ and $M$. Moreover, $Y_M$ is also well defined, but we don't know whether it is finite. One can design algorithms to calculate $Y_{\mathcal{T}}$ by minimizing $\lambda$ in $\mathcal{M}_{\mathcal{T}}$. However, we don't know how to design algorithms to get $Y_M$.

\begin{lemma}
\label{lemma-const-curv-equl-cirtical-point}
Let $r\in\mathcal{M}_{\mathcal{T}}$ be a real ball packing. Then $r$ has constant combinatorial scalar curvature if and only if it is a critical point of the CRG-functional $\lambda(r)$.
\end{lemma}
\begin{proof}
This can be shown easily from
\begin{equation}
\partial_{r_i}\lambda=\frac{K_i-\lambda}{\|r\|_{l^1}}.
\end{equation}
\end{proof}

By Lemma \ref{Lemma-Lambda-positive}, the constant curvature ball packings are isolated in
$\mathcal{M}_{\mathcal{T}}\cap \{r\in\mathds{R}^N: \sum_{i=1}^Nr_i=1\}$.
Equivalently, except for a scaling of radii, one cannot deform a ball packing continuously so that its curvature maintains constant. Recently, Xu \cite{Xu} proved the ``\emph{global rigidity}" of ball packings, i.e. the curvature map $K:V\to \mathds{R}^N, i\mapsto K_i$ is injective if one ignores the scalings of ball radii.
Thus a ball packing is determined by its curvature up to scaling. As a consequence, the ball packing with constant curvature (if it exists) is unique up to scaling. Note Xu's global rigidity can be derived from Theorem \ref{thm-extend-xu-rigid} directly. Until we proved Theorem \ref{thm-extend-xu-rigid}, we will not assume the global rigidity of $K$ a priori to prove the results in this paper so as to make the paper self-contained. Xu's global rigidity shows the uniqueness of the packing with constant curvature. The Combinatorial Yamabe Problem asks whether there exists a ball packing with constant curvature and how to find it (if it exists)? We give a glimpse into this problem with the help of the CRG-functional and the combinatorial Yamabe invariant.

\begin{theorem}\label{Thm-Q-min-iff-exist-const-curv-metric}
Consider the following four descriptions:
\begin{description}
  \item[(1)] There exists a real ball packing $\hat{r}$ with constant curvature.
  \item[(2)] The CRG-functional $\lambda(r)$ has a local minimum in $\mathcal{M}_{\mathcal{T}}$.
  \item[(3)] The CRG-functional $\lambda(r)$ has a global minimum in $\mathcal{M}_{\mathcal{T}}$.
  \item[(4)] The combinatorial Yamabe invariant $Y_{\mathcal{T}}$ is attainable by some real ball packing.
\end{description}
Then $(3)\Leftrightarrow(4)\Rightarrow(1)\Leftrightarrow(2)$. As a consequence, we get $\lambda(\hat{r})\geq Y_{\mathcal{T}}$ for any ball packing $\hat{r}$ with constant curvature.
\end{theorem}
\begin{remark}
Later we will prove ``$(2)\Rightarrow(3)$". Then it follows that $\lambda(\hat{r})=Y_{\mathcal{T}}$ for any real ball packing $\hat{r}$ with constant curvature.
\end{remark}
\begin{proof}
(3) and (4) say the same thing. They both imply (2). We prove $(1)\Leftrightarrow(2)$ below.

$(1)\Rightarrow(2)$: Let $\hat{r}\in \mathcal{M}_{\mathcal{T}}$ be a real ball packing with constant curvature $c=K_i(\hat{r})$ for all $i\in V$. Since $K$ is scaling invariant, we may assume $\|\hat{r}\|_{l^1}=1$. Consider the following functional
$$\mathcal{S}_c=\mathcal{S}-c\sum_{i=1}^Nr_i=\sum_{i=1}^N(K_i-c)r_i.$$
By the Schl\"{a}fli formula $\sum_{j\thicksim i}l_{ij}d\beta_{ij}=0$, or by Lemma \ref{Lemma-Lambda-semi-positive}, we obtain $\partial_{r_i}\mathcal{S}_c=K_i-c$. This implies that $\hat{r}$ is a critical point of the functional $\mathcal{S}_c$. Further note $Hess_r\mathcal{S}_c=\Lambda$, then it follows that $\mathcal{S}_c$ is strictly convex when restricted to the hyperplane $\{r\in\mathds{R}^N: \|r\|_{l^1}=1\}$. Hence $\hat{r}$ is a local minimum point.

$(2)\Rightarrow(1)$: Assume $\hat{r}\in \mathcal{M}_{\mathcal{T}}$ is a local minimum point of the CRG-functional $\lambda(r)$, then it is a critical point of $\lambda(r)$. Let $\hat{K}$ be the curvature at $\hat{r}$, and $\hat{\lambda}$ be the CRG-functional at $\hat{r}$. From $\partial_{r_i}\lambda=\|r\|^{-1}_{l^1}(K_i-\lambda)$, we see $\hat{K}_i=\hat{\lambda}$ for every $i\in V$. Hence $\hat{r}$ is a real ball packing with constant curvature.
\end{proof}

The above $(1)$, $(2)$, $(3)$ and $(4)$ are all equivalent. Indeed, we will prove ``$(1)\Rightarrow(3)$" in Section \ref{subsection-converg-to-const} (see Corollary \ref{corollary-Q-2}). The global ridigity ($K$ is globally injective up to scaling) tells a global result, however, we cannot use this directly to derive that a local minimum point $\hat{r}$ is a global one. The difficulty comes from that the combinatorial conformal class $\mathcal{M}_{\mathcal{T}}$ is not convex. It is the main trouble to make a local result global. Let us take a look at the procedures from Guo \cite{Guo} to Luo \cite{L2}, from Luo \cite{L1} to Bobenko, Pincall and Springborn \cite{Bobenko}, from Cooper and Rivin \cite{CR} to Xu \cite{Xu}, the results of which are all form local to global. The former says the curvature map $K$, in three different settings, is locally injective, while the later says $K$ is globally injective. These works are all based on an extension technique, which will be formulated carefully in Section \ref{section-extend-K}.

Assume the equivalence between $(1)$, $(2)$, $(3)$ and $(4)$. If any one happens, then the real ball packing $\hat{r}$ with constant curvature (if exists) is the unique (up to scaling) minimum of the CRG-functional $\lambda(r)$. One can design algorithms to minimize $\lambda(r)$, or to minimize $\mathcal{S}(r)$ under the constraint condition $\sum_{i\in V}r_i=1$. If $\lambda$ really has a global minimum in $\mathcal{M}_{\mathcal{T}}$, then this minimum is exactly $Y_{\mathcal{T}}$, and as a consequence, the Combinatorial Yamabe Problem is solvable. Otherwise, the Combinatorial Yamabe Problem has no solution.



\section{A combinatorial Yamabe flow with normalization}
\subsection{A normalization of Glickenstein's flow}
Set $u_i=\ln r_i$, and $u=(u_1,\cdots, u_N)$. Then Glickenstein's flow (\ref{Def-Flow-Glickenstein}) can be abbreviated as an ODE $\dot{u}=-K$.
The critical point (also called stable point) of this ODE is a real packing $r^*$ with $K(r^*)=0$.
Thus it makes good sense to use Glickenstein's flow (\ref{Def-Flow-Glickenstein}) if one wants to deform the combinatorial scalar curvatures to zero.
However, Glickenstein's flow (\ref{Def-Flow-Glickenstein}) is not appropriate to deform the combinatorial scalar curvature to a general constant $\mathcal{S}/\|r\|_{l^1}$ (recall we have shown that, if $K_i\equiv c$ where $c$ is a constant, then $c$ equals to $\mathcal{S}/\|r\|_{l^1}$). It is suitable to consider the following normalization of Glickenstein's flow.

\begin{definition}
Given a triangulated 3-manifold $(M, \mathcal{T})$. The normalized combinatorial Yamabe flow is
\begin{equation}\label{Def-r-normal-Yamabe flow}
\frac{dr_i}{dt}=(\lambda-K_i)r_i.
\end{equation}
\end{definition}

The normalized combinatorial Yamabe flow (\ref{Def-r-normal-Yamabe flow}) owes to Glickenstein. It was first introduced in Glickenstein's thesis \cite{G0}. 
With the help of the coordinate change $u_i=\ln r_i$, we rewrite (\ref{Def-r-normal-Yamabe flow}) as the following autonomous ODE system
\begin{equation}\label{Def-u-normal-Yamabe flow}
\frac{du_i}{dt}=\lambda-K_i.
\end{equation}

We explain the meaning of ``normalization". This means that the normalized flow (\ref{Def-r-normal-Yamabe flow}) and the un-normalized flow (\ref{Def-Flow-Glickenstein}) differ only by a change of scale in space. Let $t, r, K$ denote the variables for the flow (\ref{Def-Flow-Glickenstein}), and $t, \tilde{r}, \tilde{K}$ for the flow (\ref{Def-r-normal-Yamabe flow}).
Suppose $r(t), t\in [0,T),$ is a solution of (\ref{Def-Flow-Glickenstein}). Set
$\tilde{r}(t)=\varphi(t)r(t)$, where
$$\varphi(t)=e^{\int_0^t\lambda(r(s))ds}.$$
Then we have $\tilde{K}(t)=K(t)$ and $\tilde{\lambda}(t)=\lambda(t)$. Then it follows $\frac{d\tilde{r_i}}{dt}=(\tilde{\lambda}-\tilde{K_i}(t))\tilde{r_i}(t)$.
Conversely, if $\tilde{r}(t), t\in [0,T),$ is a solution of (\ref{Def-r-normal-Yamabe flow}), set
$r(t)=e^{-\int_0^t\tilde{\lambda}(\tilde{r}(s))ds}\tilde{r}(t)$. Then it is easy to check that $d r_i/dt=-K_ir_i$.

In the space $\mathcal{M}_{\mathcal{T}}$, the coefficient $r_i(\lambda-K_i)$ as a function of $r=(r_1,\cdots,r_N)$ is smooth and hence
locally Lipschitz continuous. By Picard theorem in classical ODE theory, the normalized flow (\ref{Def-r-normal-Yamabe flow}) has
a unique solution $r(t)$, $t\in[0,\epsilon)$ for some $\epsilon>0$. As a consequence, we have
\begin{proposition}
Given a triangulated 3-manifold $(M, \mathcal{T})$, for any initial packing $r(0)\in\mathcal{M}_{\mathcal{T}}$, the solution $\{r(t):0\leq t<T\}\subset\mathcal{M}_{\mathcal{T}}$ to the normalized flow (\ref{Def-r-normal-Yamabe flow})
uniquely exists on a maximal time interval $t\in[0, T)$ with $0<T\leq+\infty$.
\end{proposition}

We give some other elementary properties related to the normalized flow (\ref{Def-r-normal-Yamabe flow}).

\begin{proposition}\label{prop-V-descending}
Along the normalized flow (\ref{Def-r-normal-Yamabe flow}), $\|r(t)\|_{l^1}$ is invariant. Both the Cooper-Rivin functional $\mathcal{S}(r)$
and the CRG-functional $\lambda(r)$ are descending.
\end{proposition}
\begin{proof}
The conclusions follow by direct calculations $d\|r(t)\|_{l^1}/dt=0$ and
\begin{align*}
\frac{d\lambda(r(t))}{dt}&=\sum_i\partial_{r_i}\lambda(r(t))\frac{dr_i(t)}{dt}\\
&=\sum_ir_i(\lambda-K_i)\|r\|_{l^1}^{-1}(K_i-\lambda)\\
&=-\|r\|_{l^1}^{-1}\sum_ir_i(K_i-\lambda)^2\leq0.
\end{align*}
\end{proof}

\begin{proposition}\label{prop-G-bound}
The CRG-functional $\lambda(r)$ is uniformly bounded (by the information of the triangulation $\mathcal{T}$). The Cooper-Rivin functional $\mathcal{S}(r)$ is bounded along the normalized flow (\ref{Def-r-normal-Yamabe flow}).
\end{proposition}
\begin{proof}
There is a constant $c(\mathcal{T})>0$, depending only on the information of $M$ and $\mathcal{T}$, such that for all $r\in \mathcal{M}_{\mathcal{T}}$,
$$|\lambda(r)|\leq\|K\|_{l^{\infty}}\leq c(\mathcal{T}).$$
Along the normalized flow (\ref{Def-r-normal-Yamabe flow}), by Proposition \ref{prop-V-descending} we have $$|\mathcal{S}(r(t))|\leq\|K\|_{l^{\infty}}\sum_ir_i(t)\leq\|K\|_{l^{\infty}}\sum_ir_i(0).$$
\end{proof}

\begin{proposition}\label{prop-converg-imply-const-exist}
Let $r(t)$ be the unique solution of the normalized flow (\ref{Def-r-normal-Yamabe flow}). If $r(t)$ exists for all time $t\geq0$ and converges to a real ball packing $r_\infty\in\mathcal{M}_{\mathcal{T}}$, then $r_\infty$ has constant combinatorial scalar curvature.
\end{proposition}
\begin{proof}
This is a standard conclusion in classical ODE theory. Here we prove it directly. Since all $K_i(r)$ are continuous functions of $r$, $K_i(r(t))$ converges to $K_i(r_\infty)$ as $t$ goes to infinity. From Proposition \ref{prop-V-descending} and Proposition \ref{prop-G-bound}, the CRG-functional $\lambda(r(t))$ is descending and bounded below, hence converges to $\lambda(r_\infty)$.
By the mean value theorem of differential, there is a sequence of times $t_n\uparrow+\infty$, such that
$$u_i(n+1)-u_i(n)=u_i'(t_n)=\lambda(r(t_n))-K_i(r(t_n)).$$
Hence $K_i(r(t_n))\to\lambda(r_\infty)$. This leads to $K_i(r_\infty)=\lambda(r_\infty)$ for each $i\in V$. Hence $r_\infty$ has constant combinatorial scalar curvature.
\end{proof}
\begin{remark}
If the normalized combinatorial Yamabe flow (\ref{Def-r-normal-Yamabe flow}) converges, then Proposition \ref{prop-converg-imply-const-exist} says that the Combinatorial Yamabe Problem is solvable.
\end{remark}


\subsection{Singularities of the solution}
Let $\{r(t):0\leq t<T\}$ be the unique solution to the normalized flow (\ref{Def-r-normal-Yamabe flow}) on a right maximal time interval $[0, T)$ with $0<T\leq+\infty$. If the solution $r(t)$ do not converge, we call $r(t)$ develops singularities at time $T$. By Proposition \ref{prop-converg-imply-const-exist}, if there exists no ball packing with constant curvature, then $r(t)$ definitely develops singularities at $T$.
Numerical simulations show that the solution $r(t)$ may develop singularities even when the constant curvature ball packings exist. To study the long-term existence and convergence of the solutions of the normalized flow (\ref{Def-r-normal-Yamabe flow}), we need to classify the solutions according to the singularities it develops.

Intuitively, when singularities develops, $r(t)$ touches the boundary of $\mathcal{M}_{\mathcal{T}}$ as $t\uparrow T$. Roughly speaking, the boundary of $\mathcal{M}_{\mathcal{T}}$ can be classified into three types. The first type is ``\emph{0 boundary}". $r(t)$ touches the ``0 boundary" means that there exists a sequence of times $t_n\uparrow T$ and a vertex $i\in V$ so that $r_i(t_n)\to 0$. The second type is ``$+\infty$ boundary". $r(t)$ touches the ``\emph{$+\infty$ boundary}" means that there exists $t_n\uparrow T$ and a vertex $i\in V$ so that $r_i(t_n)\to +\infty$. The last type is ``\emph{tetrahedron collapsing boundary}". For this case, there exists $t_n\uparrow T$ and a tetrahedron $\{ijkl\}\in T$, such that the inequality $Q_{ijkl}>0$ does not hold any more as $n\to+\infty$. At first glance the limit behavior of $r(t)$ as $t\uparrow T$ may be mixed of the three types and may be very complicated. We show that in any finite time interval, $r(t)$ never touches the ``0 boundary" and ``$+\infty$ boundary".

\begin{proposition}
\label{prop-no-finite-time-I-singula}
The normalized flow (\ref{Def-r-normal-Yamabe flow}) will not touch the ``0 boundary" and ``$+\infty$ boundary" in finite time.
\end{proposition}

\begin{proof}
Note for all vertex $i\in V$, $|\lambda-K_i|$ are uniformly bounded by a constant $c(\mathcal{T})>0$, which depends only on the information of the triangulation. Hence
$$r_i(0)e^{-c(\mathcal{T})t}\leq r_i(t)\leq r_i(0)e^{c(\mathcal{T})t},$$
which implies $r_i(t)$ can not go to zero or $+\infty$ in finite time.
\end{proof}

Using Glickenstein's monotonicity condition \cite{G2}, which reads as
\begin{equation}\label{Def-monotonicity}
r_i\le r_j ~\mbox{ if and only if } K_i\le K_j
\end{equation}
for a tetrahedron $\{ijkl\}$, we can prove the following proposition which is essentially due to Glickenstein.

\begin{proposition}
\label{prop-glick-mc-condition}
Consider the normalized flow (\ref{Def-r-normal-Yamabe flow}) on a given $(M^3,\mathcal{T})$, assume the maximum existence time is $T$. If for all $t\in [0,T)$ and each tetrahedron, $r(t)$ satisfies the monotonicity condition (\ref{Def-monotonicity}), then $T=\infty$.
\end{proposition}
\begin{proof}
We argue by contradiction. Assume $T<\infty$. By Proposition \ref{prop-no-finite-time-I-singula}, the flow (\ref{Def-r-normal-Yamabe flow}) will not touch the ``0 boundary" and ``$+\infty$ boundary" in finite time. So we only need to get rid of $r(t)$ touching the ``tetrahedron collapsing boundary" case. We use the assumption (\ref{Def-monotonicity}). We follow the method used in \cite{G2}. We just need to show $Q_{ijkl}>0$ for every tetrahedron $\{ijkl\}$. Denote $Q=Q_{ijkl}$ without fear of confusion. To show $Q_{ijkl}>0$, we only need to show that $Q=0$ implies $\frac{dQ}{dt}> 0$. By direct calculation,
 \begin{equation}\label{derivative_of_Q}
 \frac{\partial Q}{\partial r_i}=-\frac{2}{r_i^2}\left(\frac{1}{r_j}+\frac{1}{r_k}+\frac{1}{r_l}-\frac{1}{r_i}\right).
 \end{equation}
 Then arguing as \cite{G2},
 we have $$2Q=-\left(\frac{\partial Q}{\partial r_i} r_i+\frac{\partial Q}{\partial r_j} r_j+\frac{\partial Q}{\partial r_k} r_k+\frac{\partial Q}{\partial r_l} r_l\right).$$
If $Q=0$, then
$$\frac{\partial Q}{\partial r_i} r_i+\frac{\partial Q}{\partial r_j} r_j+\frac{\partial Q}{\partial r_k} r_k+\frac{\partial Q}{\partial r_l} r_l=0.$$
Along the normalized flow (\ref{Def-r-normal-Yamabe flow}), we have
\begin{align*}
\frac{dQ}{dt}&=\frac{\partial Q}{\partial r_i} \frac{dr_i}{dt}+\frac{\partial Q}{\partial r_j} \frac{dr_j}{dt}+\frac{\partial Q}{\partial r_k} \frac{dr_k}{dt}+\frac{\partial Q}{\partial r_l} \frac{dr_l}{dt}\\
&=\frac{\partial Q}{\partial r_i}r_i(Y_{\mathcal{T}}-K_i)+\frac{\partial Q}{\partial r_j}r_j(Y_{\mathcal{T}}-K_j)+\frac{\partial Q}{\partial r_k} r_k(Y_{\mathcal{T}}-K_k)+\frac{\partial Q}{\partial r_l}r_l(Y_{\mathcal{T}}-K_l)\\
&=-\left(\frac{\partial Q}{\partial r_i} K_i r_i+\frac{\partial Q}{\partial r_j}K_j r_j+\frac{\partial Q}{\partial r_k} K_k r_k+\frac{\partial Q}{\partial r_l} K_l r_l\right)\\
&=-\left(\frac{\partial Q}{\partial r_j}(K_j-K_i) r_j+\frac{\partial Q}{\partial r_k}( K_k-K_i) r_k+\frac{\partial Q}{\partial r_l} (K_l-K_i) r_l\right).
\end{align*}
If $r_i$ is the minimum, then by (\ref{derivative_of_Q}), $\frac{\partial Q}{\partial r_j}< 0$ for $j\ne i$.
So the assumption (\ref{Def-monotonicity}) implies that $\frac{dQ}{dt}\ge 0$ for $Q=0$,  and $\frac{dQ}{dt}=0$ if and only if
$$K_i=K_j=K_k=K_l.$$
Using the assumption (\ref{Def-monotonicity}) again, we have $r_i=r_j=r_k=r_l$, but in this case $Q=\frac{8}{r_i^2}>0$. So $\frac{dQ}{dt}>0$ at  $Q=0$, which is a contradiction. Thus we have $T=\infty$.
\end{proof}

\subsection{Small energy convergence}
Let $\{r(t)\}_{t\geq0}$ be the unique solution to the normalized flow (\ref{Def-r-normal-Yamabe flow}). We call it \emph{nonsingular}
if $\{r(t)\}_{t\geq0}$ is compactly supported in $\mathcal{M}_{\mathcal{T}}$. A nonsingular solution implies that there exists a ball packing with constant curvature in $\mathcal{M}_{\mathcal{T}}$. Furthermore, the nonsingular solution converges to a ball packing with constant curvature. In this subsection, we shall prove this fact with the help of a stability result,{ which says that if the initial ball packing $r(0)$ is very close to a ball packing with constant curvature, then $r(t)$ converges to a ball packing with constant curvature. } It is obvious that if the metrics are close then their energy $\lambda$ would be close. Inversely, we will show that if the energy $\lambda(r(0))$ is close to the energy of a constant curvature metric, then the flow would converge exponentially to a constant curvature metric, which we call it "small energy convergence" or "energy gap". In fact, we will introduce a combinatorial invariant $\chi(\hat{r},\mathcal{T})$ to give a quantitative description of the smallness, see also Theorem \ref{Thm-xi-invariant-imply-converg}.

\begin{lemma}\label{Lemma-ODE-asymptotic-stable}
(\cite{P1}) Let $\Omega\subset \mathds{R}^n$ be an open set, $f\in C^1(\Omega,\mathds{R}^n)$. Consider an autonomous ODE system
$$\dot{x}=f(x),~~~x\in\Omega.$$
Assuming $x^*\in\Omega$ is a critical point of $f$, i.e. $f(x^*)=0$. If all the eigenvalues of $Df(x^*)$ have negative real part, then $x^*$ is an asymptotically stable point. More specifically, there exists a neighbourhood $U\subset \Omega$ of $x^*$, such that for any initial $x(0)\in U$, the solution $x(t)$ to equation $\dot{x}=f(x)$ exists for all time $t\geq0$ and converges exponentially fast to $x^*$.
\end{lemma}

\begin{lemma}\label{Thm-3d-isolat-const-alpha-metric}
(Stability of critical metric)
Given a triangulated manifold $(M^3, \mathcal{T})$, assume $\hat{r}\in\mathcal{M}_{\mathcal{T}}$ is a ball packing with constant curvature, then $\hat{r}$ is an asymptotically stable point of the normalized flow (\ref{Def-r-normal-Yamabe flow}). Thus if the initial real ball packing $r(0)$ deviates from $\hat{r}$ not so much, the solution $\{r(t)\}$ to the normalized flow (\ref{Def-r-normal-Yamabe flow}) exists for all time $t\geq 0$ and converges exponentially fast to the constant curvature packing $c\hat{r}$, where $c>0$ is some constant so that $c\|\hat{r}\|_{l^1}=\|r(0)\|_{l^1}$.
\end{lemma}
\begin{proof}
Set the right hand side of the flow (\ref{Def-r-normal-Yamabe flow}) as $\Gamma_i(r)=(\lambda-K_i)r_i$, $1\leq i\leq N$. Then the normalized flow (\ref{Def-r-normal-Yamabe flow}) can be written as $\dot{r}=\Gamma(r)$, which is an autonomy ODE system. Differentiate $\Gamma(r)$ at $r^*$,
\begin{equation*}
D_r\Gamma|_{r^*}=
\frac{\partial(\Gamma_{1},\cdots,\Gamma_{N})}{\partial(r_{1},\cdots,r_{N})}\Bigg|_{r^*}=
\left(
\begin{array}{ccccc}
 {\frac{\partial \Gamma_1}{\partial r_1}}& \cdot & \cdot & \cdot &  {\frac{\partial\Gamma_1}{\partial r_N}} \\
 \cdot & \cdot & \cdot & \cdot & \cdot \\
 \cdot & \cdot & \cdot & \cdot & \cdot \\
 \cdot & \cdot & \cdot & \cdot & \cdot \\
 {\frac{\partial \Gamma_N}{\partial r_1}}& \cdot & \cdot & \cdot &  {\frac{\partial \Gamma_N}{\partial r_N}}
\end{array}
\right)_{\hat{r}}=-\left(\Sigma\Lambda\right)|_{\hat{r}},
\end{equation*}
where $\Sigma=diag\{r_1,\cdots,r_N\}$. Because
$$-\Sigma\Lambda=-\Sigma^{\frac{1}{2}}\Sigma^{\frac{1}{2}}\Lambda \Sigma^{\frac{1}{2}}\Sigma^{-\frac{1}{2}}\sim \Sigma^{\frac{1}{2}}\Lambda \Sigma^{\frac{1}{2}},$$
$-\Sigma\Lambda$ has an eigenvalue $0$ and $N-1$ negative eigenvalues. Note the normalized flow (\ref{Def-r-normal-Yamabe flow}) is scaling invariant, which means any scaling $cr(t)$ ($c>0$ is a constant) of the solution $r(t)$ is also a solution to (\ref{Def-r-normal-Yamabe flow}) (perhaps with different initial value). Hence we may consider the eigenvalues of $-\Sigma\Lambda$ along (\ref{Def-r-normal-Yamabe flow}) are all negative. By Lemma \ref{Lemma-ODE-asymptotic-stable}, $\hat{r}$ is an asymptotically stable point of the normalized flow (\ref{Def-r-normal-Yamabe flow}). Hence we get the conclusion above.
\end{proof}


\begin{theorem}\label{Thm-nosingular-imply-converg}
Given a triangulated manifold $(M^3, \mathcal{T})$, assume the normalized flow (\ref{Def-r-normal-Yamabe flow}) has a nonsingular solution $r(t)$, then there is a ball packing with constant curvature. Furthermore, $r(t)$ converges exponentially fast to a ball packing with constant curvature as $t$ goes to infinity.
\end{theorem}
\begin{proof}
Since $r(t)$ is nonsingular, the right maximal existence time of $r(t)$ is $T=+\infty$. Set $\lambda(t)=\lambda(r(t))$. One can easily distinguish whether $\lambda$ is a function of $r$ or is a function of $t$ in the context without leading to any confusion. In the proof of Proposition \ref{prop-V-descending} we have derived
\begin{equation}\label{S't}
\lambda'(t)=-\frac{1}{\|r\|_{l^1}}\sum_ir_i(K_i-\lambda)^2\leq0.
\end{equation}
By Proposition \ref{prop-G-bound}, $\lambda(t)$ is bounded from below and hence converges to a number $\lambda(+\infty)$. Hence there exists a sequence $t_n\uparrow +\infty$ such that $\lambda'(t_n)\rightarrow 0$. Nonsingular solution means that $\overline{\{r(t)\}}$ stays in a compact subset of $\mathcal{M}_{\mathcal{T}}$. There is a ball packing $r_{\infty}\in\mathcal{M}_{\mathcal{T}}$, and a subsequence of $t_n$, which is still denoted as $t_n$, such that $r(t_n)\rightarrow r_{\infty}$. By (\ref{S't}), $\lambda'(t)$ is a continuous function of $r$. It follows that $\lambda'(t_n)\rightarrow \lambda'(r_{\infty})$ and then $\lambda'(r_{\infty})=0$. Substituting $\lambda'(r_{\infty})=0$ into (\ref{S't}), we see that the ball packing $r_{\infty}$ has constant curvature. Moreover, for some sufficient big $t_{n_0}$, $r(t_{n_0})$ is very close to $r_{\infty}$. Then by the Lemma \ref{Thm-3d-isolat-const-alpha-metric}, the solution $\{r(t)\}_{t\ge t_{n_0}}$ converges exponentially fast to $r_{\infty}$. Hence the original solution $\{r(t)\}_{t\ge 0}$ converges exponentially fast to $r_{\infty}$ too, which is a ball packing with constant curvature.
\end{proof}

\begin{definition}
\label{Def-chi-invariant}
Given a triangulated manifold $(M^3, \mathcal{T})$, assume there exists a real ball packing $\hat{r}$ with constant curvature. We introduce a combinatorial invariant with respect to the triangulation $\mathcal{T}$ as
\begin{equation}
\chi(\hat{r},\mathcal{T})=\inf\limits_{\gamma\in{\mathbb{S}}^{N-1};\|\gamma\|_{l^1}=0\;}\sup\limits_{0\leq t< a_\gamma}\lambda(\hat{r}+t\gamma),
\end{equation}
where $a_{\gamma}$ is the least upper bound of $t$ such that $\hat{r}+t\gamma\in \mathcal{M}_{\mathcal{T}}$ for all $0\le t<a_\gamma$.
\end{definition}


Let us denote $\mathcal{M}_{\hat{r}}$ by the star-shaped subset
$\{\hat{r}+t\gamma \in \mathcal{M}_{\mathcal{T}}: \text{ for } 0\leq t<a_\gamma \text{ and } \gamma \in \mathbb{S}^{N-1} \text{ such that $\|\gamma\|_{l^1}=0$ }\}$ of the hyperplane $\{r:\|r\|_{l^1}=\|\hat{r}\|_{l^1}\}$, where $a_\gamma$ is defined as the above definition. Since $\mathcal{M}_{\mathcal{T}}$ is not convex, the subset $\mathcal{M}_{\hat{r}}$ might be not equal to $\mathcal{M}_{\mathcal{T}}\cap \{r: \|r\|_{l^1}=\|\hat{r}\|_{l^1}\}$. By the convexity of the extended functional $\tilde{\lambda}$ in Theorem \ref{thm-tuta-S-C1-convex} and the scaling invariant of $\lambda$, we see
\begin{equation}
\label{jiang-observe}
\lambda(r)\ge \chi(\hat{r},\mathcal{T}).
\end{equation}
for all $r\in \mathcal{M}_{\mathcal{T}}\setminus \mathcal{M}_{\hat{r}}$.
(However, we will never use the above inequality in this paper. We give it here just for a better understanding of $\chi(\hat{r},\mathcal{T})$).

Let $\delta>0$ be any number so that $B(\hat{r},\delta)$ is compactly contained in $\mathcal{M}_{\mathcal{T}}$. Consider the restricted functional $\mathcal{S}(r), \;r\in B(\hat{r},\delta)\cap \{r: ||r||_{l^1}=||\hat{r}||_{l^1}\}$
as a function of $N-1$ variables. It is strictly convex and has a unique critical point at $\hat{r}$. Hence it is strictly increasing along any segment $\hat{r}+t\xi$, $t\in[0,1]$, $\xi\in\partial B(\hat{r},\delta)\cap \{r: ||r||_{l^1}=||\hat{r}||_{l^1}\}$. Let $\mathcal{S}(r')$
be the minimum of $\{\mathcal{S}(r):r\in \partial B(\hat{r},\delta), \|r\|_{l^1}=\|\hat{r}\|_{l^1}\}$, where $r'\in\partial B(\hat{r},\delta)\cap \{r: ||r||_{l^1}=||\hat{r}||_{l^1}\}$. Then by the analysis above and Theorem \ref{Thm-Q-min-iff-exist-const-curv-metric}, it follows that
\begin{equation}
\chi(\hat{r},\mathcal{T})\geq\frac{\mathcal{S}(r')}{\|r'\|_{l^1}}>\frac{\mathcal{S}(\hat{r})}{\|\hat{r}\|_{l^1}}\geq Y_{\mathcal{T}}.
\end{equation}

At first glance, the invariant $\chi(\hat{r},\mathcal{T})$ depends on the existence of a ball packing $\hat{r}$ with constant curvature, or say, $\chi(\hat{r},\mathcal{T})$ depends on the geometric information of the ball packings. However, by the uniqueness of constant curvature packings and Theorem \ref{Thm-Q-min-iff-exist-const-curv-metric}, all information of $\hat{r}$ (such as the existence or non-existence, the uniqueness, the analytical properties) are completely determined by $M$ and $\mathcal{T}$. To this extent, $\hat{r}$ is determined by the topological information of $M$ and the combinatorial information of $\mathcal{T}$. Hence the invariant $\chi(\hat{r},\mathcal{T})$ may be considered as a pure combinatorial-topological invariant. Using the combinatorial-topological invariant $\chi(\hat{r},\mathcal{T})$, we give a sufficient condition to guarantee the long time existence and the convergence of the flow (\ref{Def-r-normal-Yamabe flow}).

\begin{theorem}[Energy gap]\label{Thm-xi-invariant-imply-converg}
Given a triangulated manifold $(M^3, \mathcal{T})$, let $\hat{r}$ be a real ball packing with constant curvature. Let the initial packing be $r(0)\in\mathcal{M}_{\mathcal{T}}$. Assume
\begin{equation}
\label{ge-observe}
\lambda(r(0))\leq\chi(\hat{r},\mathcal{T}).
\end{equation}
Then the solution $r(t)$ to (\ref{Def-r-normal-Yamabe flow}) exists for all time $t\geq 0$ and converges exponentially fast to a real packing with constant curvature.
\end{theorem}



\begin{remark}
This theorem could be considered as a ``small energy convergence". Moreover, we give a precise gap bound estimate for the ``small energy convergence". One can compare our theorem with the standard ``small energy convergence theorem" in the smooth setting, for example the Calabi flow \cite{ToWe} and the $L^2$ curvature flow \cite{Str}.
\end{remark}

\begin{proof}
Assume $\lambda(r(0))\leq\chi(\hat{r},\mathcal{T})$. Denote $\lambda(t)=\lambda(r(t))$. Recall (\ref{S't}) says
\begin{equation*}
\lambda'(t)=-\frac{1}{\|r\|_{l^1}}\sum_ir_i(K_i-\lambda)^2\leq0.
\end{equation*}
If $K_i(r(0))=\lambda(0)$ for all $i\in V$, then $r(0)$ is a real packing with constant curvature. It follow that $r(t)\equiv r(0)$ is the unique solution to the flow (\ref{Def-r-normal-Yamabe flow}). This leads to the conclusion directly. If there is a vertex $i$ so that $K_i(r(0))\neq\lambda(0)$, then $\lambda'(0)<0$. Hence then $\lambda(r(t))$ is strictly descending along the flow (\ref{Def-r-normal-Yamabe flow}) for at least a small time interval $t\in[0,\epsilon)$. Thus $r(t)$ will never touches the boundary of $\mathcal{M}_{\mathcal{T}}$ along the flow (\ref{Def-r-normal-Yamabe flow}). By classical ODE theory, the solution $r(t)$ exists for all time $t\in[0,+\infty)$. Moreover, $\{r(t)\}_{t\geq0}\subset\subset \mathcal{M}_{\mathcal{T}}$, that is, $\{r(t)\}_{t\geq0}$ is compactly supported in $\mathcal{M}_{\mathcal{T}}$.
By Theorem \ref{Thm-nosingular-imply-converg}, there exists a real ball packing $r_{\infty}$ with constant curvature so that $r(t)$ converges exponentially fast to $r_{\infty}$. Thus we get the conclusion.
\end{proof}


By this time, we have not enough knowledge to show $r_{\infty}$ is actually a scaling of $\hat{r}$, unless we acknowledge that the constant curvature packing is unique (up to scaling) which will be derived after we introduce the extension technique in Section \ref{section-extend-CR-funct}. However, by a more subtle argument, we can prove: if further assume $\lambda(r(0))<\chi(\hat{r},\mathcal{T})$, then $r(t)$ converges exponentially fast to $r_{\infty}$, which is a scaling of $\hat{r}$ (so as $\|r_{\infty}\|_{l^1}=\|r(0)\|_{l^1}$).

Note Theorem \ref{Thm-xi-invariant-imply-converg} is established under the framework that there exists a constant curvature ball packing $\hat{r}$ in $\mathcal{M}_{\mathcal{T}}$. Theorem \ref{Thm-nosingular-imply-converg} implies that, if there is no any constant curvature ball packings, then the solution $r(t)$ to (\ref{Def-r-normal-Yamabe flow}) touches the boundary of $\mathcal{M}_{\mathcal{T}}$. More specifically, we have


\begin{corollary}\label{coro-collasp-flow}
Given a triangulated manifold $(M^3, \mathcal{T})$, let $\{r(t)\}_{0\leq t<T}$ be the unique maximal solution to the normalized flow (\ref{Def-r-normal-Yamabe flow}). Assume there is no any constant curvature ball packings, then there exists a time sequence $t_n\to T$ such that
\begin{enumerate}
  \item[(1)] if $T<+\infty$, then $Q_{ijkl}(r(t_n))\to 0$ for some tetrahedron $\{ijkl\}$;
  \item[(2)] if $T=+\infty$, then either $Q_{ijkl}(r(t_n))\to 0$ for some tetrahedron $\{ijkl\}$, or $r_i(t_n)\to 0$ for some vertex $i$.
\end{enumerate}
\end{corollary}

\section{The extended curvature and functionals}
To get global convergence of the normalized Yamabe flow (\ref{Def-r-normal-Yamabe flow}), we require that its solution $r(t)$ exists for all time $t\in[0,\infty)$ at least. However, as our numerical experiments show, $r(t)$ may collapse in finite time. To prevent finite time collapsing, Glickenstein's monotonicity condition (\ref{Def-monotonicity}) seems useful, but it is too strong to be satisfied. Although (\ref{ge-observe}) guarantees the convergence of $r(t)$, it can't deal with more general case such as $r(0)\in \mathcal{M}_{\mathcal{T}}\setminus \mathcal{M}_{\hat{r}}$ by (\ref{jiang-observe}). We provide a method to extend $r(t)$ so as it always exists for all time in this section.

\subsection{Packing configurations by four tangent balls}
\label{section-extend-K}
In this section and the next Section \ref{section-extend-solid-angle}, $r=(r_1,r_2,r_3,r_4)\in\mathds{R}^4_{>0}$ means a point in $\mathds{R}^4_{>0}$.
Recall the definition of $Q_{ijkl}$ in (\ref{nondegeneracy condition}), from which we can derive the expression of $Q_{1234}$.

Let $\tau=\{1234\}$ be a combinatorial tetrahedron which contains only combinatorial information but not any geometric information. By definition, the combinatorial information of $\tau$ is a vertex set $\{1,2,3,4\}$, an edge set $\{\{12\},\{13\},\{14\},\{23\},\{24\},\{34\}\}$, a face set $\{\{123\},\{124\},\{134\},\{234\}\}$ and a tetrahedron set $\{\{1234\}\}$. For any $r=(r_1,r_2,r_3,r_4)\in\mathds{R}^4_{>0}$, endow each edge $\{ij\}$ in the edge set with an edge length $l_{ij}=r_i+r_j$. If $Q_{1234}>0$, then the six edges of $\{1234\}$ with lengths $l_{12},l_{13},l_{14},l_{23},l_{24},l_{34}$ form the edges of an Euclidean tetrahedron. In this case, we call $\tau=\{1234\}$ a \emph{real tetrahedron}.
Otherwise, $Q_{1234}\leq0$, and we call $\tau=\{1234\}$ a \emph{virtual tetrahedron} (in other words, $\tau$ degenerates). For a real tetrahedron $\tau=\{1234\}$,
denote $\alpha_{i}$ by the solid angle at each vertex $i\in\{1,2,3,4\}$. All real tetrahedrons can be considered as the following proper subset of $\mathds{R}^4_{>0}$,
\begin{equation}
\Omega_{1234}=\left\{(r_1,r_2,r_3,r_4)\in\mathds{R}^4_{>0}:Q_{1234}>0\right\}.
\end{equation}
Obviously, $\Omega^{-1}_{1234}=\{(r_1^{-1},\cdots,r_4^{-1}):(r_1,r_2,r_3,r_4)\in\Omega_{1234}\}$ is an open convex cone in $\mathds{R}^4_{>0}$. Hence $\Omega_{1234}$,
the homeomorphic image of $\Omega^{-1}_{1234}$, is simply-connected with peicewise analytic boundary.

Denote $\{i,j,k,l\}=\{1,2,3,4\}$ and place three balls $S_j$, $S_k$ and $S_l$, externally tangent to each other on the plane, with radii $r_j$, $r_k$ and $r_l>0$. Let $S_i$ be the fourth ball with radius $r_i>0$. If $r_i$ is very small and is very closed to $0$, then obviously
$$\frac{1}{r_{i}}>\frac{1}{r_{j}}+\frac{1}{r_{k}}+\frac{1}{r_{l}}+
2\sqrt{\frac{1}{r_{j}r_k}+\frac{1}{r_{j}r_l}+\frac{1}{r_{k}r_l}}.$$
Hence it follows
$$\left(\frac{1}{r_{i}}-\left(\frac{1}{r_{j}}+\frac{1}{r_{k}}+\frac{1}{r_{l}}\right)\right)^2
>4\left(\frac{1}{r_{j}r_k}+\frac{1}{r_{j}r_l}+\frac{1}{r_{k}r_l}\right)$$
and further $Q_{1234}<0$. Geometrically, the fourth ball $S_i$ goes through the gap between the other three mutually tangent balls. Let the radius $r_i$ increases gradually to one with
$$\frac{1}{r_{i}}=\frac{1}{r_{j}}+\frac{1}{r_{k}}+\frac{1}{r_{l}}+
2\sqrt{\frac{1}{r_{j}r_k}+\frac{1}{r_{j}r_l}+\frac{1}{r_{k}r_l}}.$$
By this time $Q_{1234}=0$. Geometrically, the fourth ball $S_i$ is in the gap between the other three mutually tangent balls, and is externally tangent to them all. Denote
$$f_i(r_j,r_k,r_l)=\left(\frac{1}{r_{j}}+\frac{1}{r_{k}}+\frac{1}{r_{l}}+
2\sqrt{\frac{1}{r_{j}r_k}+\frac{1}{r_{j}r_l}+\frac{1}{r_{k}r_l}}\,\right)^{-1}$$
and the $i$-th virtual tetrahedron space (abbreviated as ``\emph{$i$-th virtual space}") by
\begin{equation}
D_i=\{(r_1,r_2,r_3,r_4)\in\mathds{R}_{>0}^4:0<r_i\leq f_i(r_j,r_k,r_l)\}.
\end{equation}
Note $D_i$ is contractible and hence is simply-connected.

\begin{lemma}\label{lemma-Di-imply-ri-small}
In the $i$-th virtual space $D_i$, one have $r_i<\min\{r_j,r_k,r_l\}$.
\end{lemma}
\begin{proof}
One can get the conclusion easily from
\begin{equation*}
\frac{1}{r_{i}}\geq \frac{1}{f_{i}(r_j,r_k,r_l)}
=\frac{1}{r_{j}}+\frac{1}{r_{k}}+\frac{1}{r_{l}}+
2\sqrt{\frac{1}{r_{j}r_k}+\frac{1}{r_{j}r_l}+\frac{1}{r_{k}r_l}}.
\end{equation*}
\end{proof}

Because any two numbers of $r_1$, $r_2$, $r_3$ and $r_4$ can't be strictly minimal simultaneously, we obviously have the following corollary.
\begin{corollary}
The virtual space $D_1$, $D_2$, $D_3$ and $D_4$ are mutually disjoint.
\end{corollary}

\begin{lemma}\label{lemma-er-ze-yi}
Assume $r_i>0$ is the minimum of $r_1$, $r_2$, $r_3$ and $r_4$. Then the inequality
\begin{equation*}
\frac{1}{r_{i}}\leq\frac{1}{r_{j}}+\frac{1}{r_{k}}+\frac{1}{r_{l}}-
2\sqrt{\frac{1}{r_{j}r_k}+\frac{1}{r_{j}r_l}+\frac{1}{r_{k}r_l}}.
\end{equation*}
will never happen. In other words, if the inequality holds true, then $r_i>\min\{r_j,r_k,r_l\}$.
\end{lemma}
\begin{proof}
Assume the above inequality holds true. Note its right hand side is symmetric with respect to $r_j$, $r_k$ and $r_l$. We may assume $r_j\leq r_k\leq r_l$.
Then from $1/r_i\geq 1/r_j$, we get
$$\frac{1}{r_{k}}+\frac{1}{r_{l}}\geq
2\sqrt{\frac{1}{r_{j}r_k}+\frac{1}{r_{j}r_l}+\frac{1}{r_{k}r_l}}.$$
Taking square and note $2/r_j\geq 1/r_{k}+1/r_{l}$, we see
$$\frac{1}{r_{k}^2}+\frac{1}{r_{l}^2}-\frac{2}{r_{k}r_l}\geq
\frac{4}{r_{j}r_k}+\frac{4}{r_{j}r_l}\geq2\left(\frac{1}{r_{k}}+\frac{1}{r_{l}}\right)^2,$$
which is a contradiction.
\end{proof}

\begin{lemma}\label{lemma-ri-small-imply-Di}
If $Q_{1234}\leq0$, then $\{r_1,r_2,r_3, r_4\}$ have a strictly minimal value. Moreover, if $\{r_1,r_2,r_3, r_4\}$ attains its strictly minimal value at $r_i$ for some $i\in\{1,2,3,4\}$, then $r\in D_i$.
\end{lemma}
\begin{proof}
We may assume $r_i\leq r_j\leq r_k\leq r_l$. It's easy to express $Q_{1234}\leq0$ as the following
$$\frac{1}{r_{i}^2}-\frac{2}{r_i}\left(\frac{1}{r_{j}}+\frac{1}{r_{k}}+\frac{1}{r_{l}}\right)\geq
\left(\frac{1}{r_{j}}+\frac{1}{r_{k}}+\frac{1}{r_{l}}\right)^2-2\left(\frac{1}{r_{j}^2}+\frac{1}{r_{k}^2}+\frac{1}{r_{l}^2}\right).$$
Solving the above inequality, we get either
$$\frac{1}{r_{i}}\leq\frac{1}{r_{j}}+\frac{1}{r_{k}}+\frac{1}{r_{l}}-
2\sqrt{\frac{1}{r_{j}r_k}+\frac{1}{r_{j}r_l}+\frac{1}{r_{k}r_l}}$$
or
$$\frac{1}{r_{i}}\geq\frac{1}{r_{j}}+\frac{1}{r_{k}}+\frac{1}{r_{l}}+
2\sqrt{\frac{1}{r_{j}r_k}+\frac{1}{r_{j}r_l}+\frac{1}{r_{k}r_l}}.$$
The first case will never happen by Lemma \ref{lemma-er-ze-yi}. Obviously, the second case implies
$r_i<\min\{r_j,r_k,r_l\}$, and in this case $r\in D_i$ obviously. 
\end{proof}

From Lemma \ref{lemma-Di-imply-ri-small} and Lemma \ref{lemma-ri-small-imply-Di}, we derive that, under the assumption $Q_{1234}\leq0$, the four radius $r_1$, $r_2$, $r_3$ and $r_4$ have a strict minimal value. Moreover, $r_i$ is a strictly minimum if and only if $r$ lies in the $i$-th virtual space $D_i$. This fact leads to the following observation
\begin{equation}
\mathds{R}_{>0}^4-\Omega_{1234}=D_1\,\dot{\cup}\,D_2\,\dot{\cup}\,D_3\,\dot{\cup}\,D_4,
\end{equation}
where the symbol ``$\dot{\cup}$" means ``disjoint union". As a consequence, we can classify all (real or virtual) tetrahedrons $r\in\mathds{R}_{>0}^4$ as follows:
\begin{itemize}
  \item If $Q_{1234}(r)>0$, then $r\in \Omega_{1234}$ and $r$ makes $\{1234\}$ a real tetrahedron.
  \item If $Q_{1234}(r)\leq0$, then $r$ is virtual packing. At this time,
          \begin{itemize}
            \item either $$\frac{1}{r_{i}}\geq\frac{1}{r_{j}}+\frac{1}{r_{k}}+\frac{1}{r_{l}}+
                  2\sqrt{\frac{1}{r_{j}r_k}+\frac{1}{r_{j}r_l}+\frac{1}{r_{k}r_l}}.$$
                  in this case, $r_i$ is the strictly minimum and hence $r\in D_i$;
            \item or $$\frac{1}{r_{i}}\leq \frac{1}{r_{j}}+\frac{1}{r_{k}}+\frac{1}{r_{l}}-
                  2\sqrt{\frac{1}{r_{j}r_k}+\frac{1}{r_{j}r_l}+\frac{1}{r_{k}r_l}}.$$
                  in this case, $r_i>\min\{r_j,r_k,r_l\}$ by Lemma \ref{lemma-er-ze-yi}. Moreover, since the right hand side of the above inequality is positive, we further get
                     \begin{itemize}
                       \item either $$\frac{1}{\sqrt{r_l}}>\frac{1}{\sqrt{r_j}}+\frac{1}{\sqrt{r_k}}.$$
                             in this case, $r_l$ is the strictly minimum and hence $r\in D_l$;
                       \item or $$\frac{1}{\sqrt{r_l}}<\Big|\frac{1}{\sqrt{r_j}}-\frac{1}{\sqrt{r_k}}\Big|.$$
                             in this case, one can show $r_j\neq r_k$ and further
                               \begin{itemize}
                                 \item if $r_j>r_k$, then $r_k$ is the strictly minimum and hence $r\in D_k$;
                                 \item if $r_j<r_k$, then $r_j$ is the strictly minimum and hence $r\in D_j$.
                               \end{itemize}
                     \end{itemize}
          \end{itemize}
\end{itemize}

\subsection{A $C^0$-extension of the solid angles}
\label{section-extend-solid-angle}
The solid angle is initially defined for real tetrahedrons. There is a natural way to extend its definition to even virtual tetrahedrons. We explain this procedure in this section.

It seems that Bobenko, Pinkall, Springborn \cite{Bobenko} first introduced the extension methods. If an Euclidean or hyperbolic triangle degenerates (that is, the three side lengthes still positive, but do not satisfy the triangle inequalities anymore), the angle opposite the side that is too long is defined to be $\pi$, while the other two angles are defined to be $0$. Using this method, they established a variational principle connecting surprisingly Milnor's Lobachevsky volume function of decorated hyperbolic ideal tetrahedrons and Luo's discrete conformal changes \cite{L1}. Luo \cite{L2} systematically developed their extension idea and proved some rigidity results related to inversive distance circle packings and discrete conformal factors. See \cite{GJ1}-\cite{GJ4} for more example. The extension of dihedral angles in a $3$-dimensional decorated ideal (or hyper-ideal) hyperbolic polyhedral first appeared in Luo and Yang's work \cite{LuoYang}. They proved the rigidity of hyperbolic cone metrics on $3$-manifolds which are isometric gluing of ideal and hyper-ideal tetrahedra in hyperbolic spaces.

As to the conformal tetrahedron configured by four ball packings, Xu \cite{Xu} gave a natural extension of solid angles. More precisely, if $r_1$, $r_2$, $r_3$ and $r_4$ satisfy $Q_{1234}>0$, then the real tetrahedron $\{1234\}$ is embedded in an Euclidean space. Denote $\alpha_i$ by the solid angle at a vertex $i\in\{1,2,3,4\}$ and define $\tilde\alpha_i=\alpha_i$. If $r_1$, $r_2$, $r_3$ and $r_4$ satisfy $Q_{1234}\leq0$, then the tetrahedron $\{1234\}$ is virtual. By Lemma \ref{lemma-ri-small-imply-Di}, $\{r_1,r_2,r_3, r_4\}$ have a strictly minimal value at a vertex $r_i$ for some $i\in\{1,2,3,4\}$, and $r\in D_i$. Geometrically, this is exactly the case that the ball $S_i$ go through the gap between the three mutually tangent balls $S_j$, $S_k$ and $S_l$.
Define $\tilde\alpha_i=2\pi$ and the other three solid angles to be $0$. By this, any real solid angle $\alpha_i$ (defined on $\Omega_{1234}$) is extended to the generalized solid angle $\tilde\alpha_i$ (defined on $\mathds{R}^4_{>0}$).
Xu (Lemma 2.6, \cite{Xu}) showed that this extension is continuous. The argument there relies on heavily geometric intuition. We give an alternative analytic proof here, which is more rigorous.

\begin{lemma}
\label{lemma-xu-extension}
For each vertex $i\in\{1,2,3,4\}$, the extended solid angle $\tilde{\alpha}_i$, defined on $\mathds{R}^4_{>0}$, is a continuous extension of $\alpha_i$.
\end{lemma}
\begin{proof}
Obviously, $\tilde{\alpha}_i$ is an extension of $\alpha_i$. It is continuous (in fact, $C^{\infty}$-smooth) in $\Omega_{1234}$ since $\alpha_i$ is. It is a constant and hence is continuous in the interior of $D_1$, $D_2$, $D_3$ and $D_4$. Fix an arbitrary point $x=(x_1,x_2,x_3,x_4)\in \partial D_i$, where the boundary is taken with respect to the topology of $\mathds{R}^4_{>0}$. By Lemma \ref{lemma-Di-imply-ri-small}, one have $x_i<\min\{x_j,x_k,x_l\}$. Choose a small open neighborhood $U_x\subset\mathds{R}^4_{>0}$ of $x$, such that $r_i<\min\{r_j,r_k,r_l\}$ for each $r\in U_x$. For any sequence $\{r^{(n)}\}\subset U_x$ with $r^{(n)}\rightarrow x$, if $r^{(n)}$ is contained in $D_i$ then $\tilde{\alpha}_i(r^{(n)})=2\pi$; if $r^{(n)}$ is not in $D_i$, then by Lemma \ref{Lemma-Glicken-bdr-converge} below, $\alpha_i(r^{(n)})$ goes to $2\pi$. Hence there always holds $\tilde{\alpha}_i(r^{(n)})\rightarrow2\pi$, which implying that $\tilde{\alpha}_i$ is continuous at $x$. Thus $\tilde{\alpha}_i$ is continuous on $\partial D_i$. Similarly, one can show that $\tilde{\alpha}_i$ is continuous on $\partial D_j$, $\partial D_k$ and $\partial D_l$. Then it follows that $\tilde{\alpha}_i$ is continuous on $\mathds{R}^4_{>0}$.
\end{proof}

\begin{lemma}\label{Lemma-Glicken-bdr-converge}
(Glickenstein \cite{G2}, Proposition 6) If $Q_{1234}\rightarrow0$ without any of the $r_i$ going to $0$, then one solid angle goes to $2\pi$ and the others go to $0$. The solid angle $\alpha_i$ which goes to $2\pi$ corresponds to $r_i$ being the minimum.
\end{lemma}

The natural extension of $\alpha$ to $\tilde{\alpha}$ is only $C^0$-continuous. The following example shows that we can't get higher regularity, such as Lipschitz continuity or H\"{o}lder continuity.
\begin{example}
\label{example-only-c-zero}
Fix $r_2=r_3=r_4=1$. Then the critical case is exactly $r_1=2/\sqrt{3}-1$. Hence the point $(2/\sqrt{3}-1,1,1,1)$ lies in $\partial D_1$. Recall Glickenstein's calculation (see the formula (7) in \cite{G1})
$$\frac{\partial\alpha_i}{\partial r_j}=\frac{4r_ir_jr_k^2r_l^2}{3P_{ijk}P_{ijl}V_{ijkl}}
\left(\frac{1}{r_i}\left(\frac{1}{r_j}+\frac{1}{r_k}+\frac{1}{r_l}\right)+
\frac{1}{r_j}\left(\frac{1}{r_i}+\frac{1}{r_k}+\frac{1}{r_l}\right)-\left(\frac{1}{r_k}-\frac{1}{r_l}\right)^2\right),$$
where $\{i,j,k,l\}=\{1,2,3,4\}$, $P_{ijk}=2(r_i+r_j+r_k)$, $P_{ijl}=2(r_i+r_j+r_l)$ and $V_{ijk}$ is the volume.  
Beacuse $V_{ijkl}=0$, we see $\partial\alpha_1/\partial r_2=+\infty$ at this point.
\end{example}

\subsection{A convex $C^1$-extension of the Cooper-Rivin functional}
\label{section-extend-CR-funct}
Now consider a triangulated manifold $(M,\mathcal{T})$, with a ball packing $r=(r_1,\cdots,r_N)\in\mathds{R}^N_{>0}$.
Recall the space of all real ball packings is
\begin{equation*}
 \mathcal{M}_{\mathcal{T}}=\left\{\;r\in\mathds{R}^N_{>0}\;:\;Q_{ijkl}>0, \;\forall \{i,j,k,l\}\in T\;\right\},
\end{equation*}
while the space of all virtual ball packings is $\mathds{R}^N_{>0}\setminus\mathcal{M}_{\mathcal{T}}$, where
\begin{equation*}
Q_{ijkl}=\left(\frac{1}{r_{i}}+\frac{1}{r_{j}}+\frac{1}{r_{k}}+\frac{1}{r_{l}}\right)^2-
2\left(\frac{1}{r_{i}^2}+\frac{1}{r_{j}^2}+\frac{1}{r_{k}^2}+\frac{1}{r_{l}^2}\right).
\end{equation*}
Recall $\alpha_{ijkl}$ is the solid angle at $i$ of an Euclidean conformal tetrahedron $\{ijkl\}$ configured by a real ball packing $r\in\mathcal{M}_{\mathcal{T}}$. Thus $\alpha_{ijkl}(r)$ is a smooth function of all real ball packings.
By Lemma \ref{lemma-xu-extension}, the solid angle $\alpha_{ijkl}$ extends continuously to an extended solid angle $\tilde{\alpha}_{ijkl}$, with the domain of definition extends from $\mathcal{M}_{\mathcal{T}}$ to $\mathds{R}^N_{>0}$. Consequently, the combinatorial scalar curvature at a vertex $i\in V$, that is
\begin{eqnarray*}
  K_i(r)=4\pi-\sum_{\{ijkl\}\in \mathcal{T}_3}\alpha_{ijkl}(r),\;r\in\mathcal{M}_{\mathcal{T}}
\end{eqnarray*}
extends continuously to the extended curvature
\begin{eqnarray}
  \widetilde K_i(r)=4\pi-\sum_{\{ijkl\}\in \mathcal{T}_3}\tilde\alpha_{ijkl}(r),\;r\in\mathds{R}^N_{>0},
\end{eqnarray}
which is defined for all $r\in\mathds{R}^N_{>0}$. Similarly, the Cooper-Rivin functional $\mathcal{S}(r)$, $r\in\mathcal{M}_{\mathcal{T}}$ in (\ref{def-cooper-rivin-funct}), can be extended naturally to the following ``\emph{extended Cooper-Rivin functional}"
\begin{eqnarray}
  \widetilde{\mathcal{S}}(r)=\sum_{i=1}^N \widetilde K_i r_i, \;r\in\mathds{R}^N_{>0}.
\end{eqnarray}
Moreover, the CRG-functional $\lambda(r)$, $r\in\mathcal{M}_{\mathcal{T}}$ can be extended naturally to
\begin{eqnarray}
  \tilde{\lambda}(r)=\frac{\sum_{i=1}^N \widetilde K_i r_i}{\sum_{i=1}^N r_i}, \;r\in\mathds{R}^N_{>0},
\end{eqnarray}
which is called the \emph{extended CRG-functional} and is uniformly bounded by the information of the triangulation $\mathcal{T}$. Since every $\widetilde{K}_i$ is continuous, both $\widetilde{\mathcal{S}}(r)$ and $\tilde{\lambda}(r)$ are continuous on $\mathds{R}^N_{>0}$. We further prove that $\widetilde{\mathcal{S}}(r)$ are in fact convex and $C^1$-smooth on $\mathds{R}^N_{>0}$.\\

We follow the approach pioneered by Luo \cite{L2}. A differential $1$-form $\omega=\sum_{i=1}^n a_i(x)dx_i$ in an open set $U\subset\mathds{R}^n$ is said to be continuous if each $a_i(x)$ is a continuous function on $U$. A continuous $1$-form $\omega$ is called closed if $\int_{\partial \tau} \omega=0$ for any Euclidean triangle $\tau\subset U$. By the standard approximation theory, if $\omega$ is closed and $\gamma$ is a piecewise $C^1$-smooth null-homologous loop in $U$, then $\int_\gamma \omega=0$. If $U$ is simply connected, then the integral $F(x)=\int_a^x\omega$ is well defined (where $a\in U$ is arbitrarily chosen), independent of the choice of piecewise smooth paths in $U$ from $a$ to $x$. Moreover, the function $F(x)$ is $C^1$-smooth so that $\frac{\partial F(x)}{\partial x_i}=a_i(x)$. Luo established the following fundamental $C^1$-smooth and convex extension theory.
\begin{lemma}
\label{lemma-luo's-extension}
(Luo's convex $C^1$-extension, \cite{L2})
Suppose $X\subset \mathds{R}^n$ is an open convex set and $A\subset X$ is an open and simply connected subset of $X$ bounded by a real analytic codimension-1 submanifold in $X$. If $\omega=\sum_{i=1}^n a_i(x)dx_i$ is a continuous closed $1$-form on $A$ so that $F(x)=\int_a^x\omega$ is locally convex on $A$ and each $a_i$ can be extended continuously to $X$ by constant functions to a function $\tilde a_i$ on $X$, then $\widetilde F(x)=\int_a^x \tilde a_i(x)dx_i$ is a $C^1$-smooth convex function on $X$ extending $F$.
\end{lemma}

Now we come back to our settings. Since
$$\frac{\partial K_i}{\partial r_j}=\frac{\partial K_j}{\partial r_i}$$
on $\mathcal{M}_\mathcal{T}$ (for example, see \cite{CR,G1}, or see Lemma \ref{Lemma-Lambda-semi-positive}), $\sum_{i=1}^N K_i dr_i$ is a closed $C^\infty$-smooth $1$-form on $\mathcal{M}_\mathcal{T}$. Note $\mathcal{M}_\mathcal{T}$ is simply connected (see \cite{CR}), hence for an arbitrarily chosen $r_0\in \mathcal{M}_{\mathcal{T}}$, the potential functional
\begin{eqnarray}\label{Def-potential}
  F(r)=\int_{r_0}^r\sum_{i=1}^NK_idr_i,\ \ r\in \mathcal{M}_{\mathcal{T}}
\end{eqnarray}
is well defined. Note that
$\nabla_r F=K=\nabla_r\mathcal{S}$, we can easily get $F(r)=\mathcal{S}(r)-\mathcal{S}(r_0)$ for each $r\in\mathcal{M}_\mathcal{T}$. By Lemma \ref{Lemma-Lambda-positive}, the potential functional (\ref{Def-potential}) is locally convex on $\mathcal{M}_\mathcal{T}$ and is strictly locally convex when restricted to the hyperplane $\{x\in\mathds{R}^N:\|x\|_{l^1}=1\}$. For each tetrahedron $\{ijkl\}\in \mathcal{T}_3$, $\alpha_{ijkl}dr_i+\alpha_{jikl}dr_j+\alpha_{kijl}dr_k+\alpha_{lijk}dr_l$ is a smooth closed $1$-form on $\mathcal{M}_\mathcal{T}$. Hence the following integration
\begin{eqnarray*}
F_{ijkl}(r)=\int_{r_0}^r \alpha_{ijkl}dr_i+ \alpha_{jikl}dr_j+ \alpha_{kijl}dr_k+ \alpha_{lijk}dr_l,\ r\in\mathcal{M}_\mathcal{T}
\end{eqnarray*}
is well defined and is a $C^{\infty}$-smooth locally concave function on $\mathcal{M}_\mathcal{T}$. By Lemma \ref{lemma-xu-extension}, each solid angle $\alpha_{ijkl}$ can be extended continuously by constant functions to a generalized solid angle $\tilde{\alpha}_{ijkl}$.
Using Luo's extension Lemma \ref{lemma-luo's-extension}, the following integration
\begin{eqnarray*}
\widetilde F_{ijkl}(r)=\int_{r_0}^r \tilde\alpha_{ijkl}dr_i+\tilde\alpha_{jikl}dr_j+\tilde\alpha_{kijl}dr_k+\tilde\alpha_{lijk}^ldr_l,\ r\in\mathds{R}^N_{>0}
\end{eqnarray*}
is well defined and $C^1$-smooth that extends $F_{ijkl}$. Moreover, $\widetilde F_{ijkl}$ is concave on $\mathds{R}^N_{>0}$. By
\begin{eqnarray*}
  \sum_{i=1}^N\widetilde K_idr_i&=&\sum_{i=1}^N\left(4\pi-\sum_{\{ijkl\}\in \mathcal{T}_3}\tilde\alpha_{ijkl}\right)dr_i\\
  &=&4\pi dr_i-\sum_{i=1}^N\sum_{\{ijkl\}\in \mathcal{T}_3}\tilde\alpha_{ijkl}dr_i\\
  &=&4\pi dr_i-\sum_{\{ijkl\}\in\mathcal{T}_3}
  \left(\tilde\alpha_{ijkl}dr_i+\tilde\alpha_{jikl}dr_j+\tilde\alpha_{kijl}dr_k+\tilde\alpha_{lijk}dr_l\right),
\end{eqnarray*}
the following integration
\begin{equation}
\widetilde F(r)=\int_{r_0}^r\sum_{i=1}^N\widetilde K_i dr_i,\ r\in\mathds{R}^N_{>0}
\end{equation}
is well defined and $C^1$-smooth that extends $F$ defined in formula (\ref{Def-potential}). Moreover, $\widetilde F(r)$ is convex on $\mathds{R}^N_{>0}$. We shall prove that the extended Cooper-Rivin functional $\widetilde{\mathcal{S}}(r)$ differs form $\widetilde F(r)$ by a constant. First we show that $\widetilde{\mathcal{S}}(r)$ is $C^1$-smooth.

\begin{theorem}
\label{thm-tuta-S-C1-convex}
The extended Cooper-Rivin functional $\widetilde{\mathcal{S}}(r)$ is convex on $\mathds{R}^N_{>0}$. Moreover,
$\widetilde{\mathcal{S}}(r)\in C^{\infty}(\mathcal{M}_\mathcal{T})\cap C^{1}(\mathds{R}^N_{>0})$. As a consequence, the extended CRG-functional $\tilde{\lambda}(r)$ is $C^1$-smooth on $\mathds{R}^N_{>0}$ and is convex when restricted to the hyperplane $\{x\in\mathds{R}^N:\|x\|_{l^1}=1\}$.
\end{theorem}
\begin{proof}
We just need to show $\widetilde{\mathcal{S}}(r)\in C^1(\mathds{R}^N_{>0})$. For each tetrahedron $\{ijkl\}\in \mathcal{T}_3$, set $$\widetilde{\mathcal{S}}_{ijkl}(r_i,r_j,r_k,r_l)=\tilde\alpha_{ijkl}r_i+\tilde\alpha_{jikl}r_j+\tilde\alpha_{kijl}r_k+\tilde\alpha_{lijk}r_l.$$
For every vertex $p\in\{i,j,k,l\}$, on the open set $\left\{(r_i,r_j,r_k,r_l)\in\mathds{R}^4_{>0}:Q_{ijkl}>0\right\}$ we get
\begin{equation}\label{formula-schlaf-extend}
\frac{\partial\widetilde{\mathcal{S}}_{ijkl}}{\partial r_p}=\tilde\alpha_{pqst}
\end{equation}
by the Schl\"{a}ffli formula (see Appendix \ref{appen-schlafi}), where $q,s,t$ are the other three vertices other than $p$. On the open domain $D_p$ where $\tilde\alpha_{pqst}=2\pi$, we have $\widetilde{\mathcal{S}}_{ijkl}=2\pi r_p$, and hence (\ref{formula-schlaf-extend}) is also valid. On the open domain $D_q$, $D_s$ or $D_t$ where $\tilde\alpha_{pqst}=0$, $\widetilde{\mathcal{S}}_{ijkl}$ equals to $2\pi r_q$, $2\pi r_s$ or $2\pi r_t$, hence we still have (\ref{formula-schlaf-extend}). By the classical Darboux Theorem in mathematical analysis, (\ref{formula-schlaf-extend}) is valid on $\mathds{R}^4_{>0}\cap\partial\left\{(r_i,r_j,r_k,r_l)\in\mathds{R}^4_{>0}:Q_{ijkl}>0\right\}$. Hence (\ref{formula-schlaf-extend}) is always true on $\mathds{R}^4_{>0}$. Because $\tilde\alpha_{pqst}$ is continuous, we see $\widetilde{\mathcal{S}}_{ijkl}$ is $C^1$-smooth on $\mathds{R}^4_{>0}$. Further by
\begin{equation*}
\widetilde{\mathcal{S}}(r)=4\pi\sum_{i=1}^Nr_i-\sum_{\{ijkl\}\in \mathcal{T}_3}\widetilde{\mathcal{S}}_{ijkl}(r_i,r_j,r_k,r_l),
\end{equation*}
we get the conclusion.
\end{proof}

\begin{corollary}\label{coro-extend-schlafi-formula}
The following extended Schl\"{a}ffli formula is valid on $\mathds{R}^N_{>0}$,
$$d\left(\sum_{i=1}^N\widetilde{K}_ir_i\right)=\sum_{i=1}^N\widetilde{K}_idr_i.$$
\end{corollary}

Corollary \ref{coro-extend-schlafi-formula} implies that $\nabla_r\widetilde{\mathcal{S}}=\widetilde{K}$. Note $\nabla_r\widetilde{F}=\widetilde{K}$ too, hence we obtain
\begin{corollary}
$\widetilde{\mathcal{S}}(r)=\widetilde F(r)+\mathcal{S}(r_0)$ on $\mathds{R}^N_{>0}$.
\end{corollary}

Denote $K(\mathcal{M}_{\mathcal{T}})$ by the image set of the curvature map $K: \mathcal{M}_{\mathcal{T}}\rightarrow \mathds{R}^N$. Using the extended Cooper-Rivin functional, we get the following

\begin{theorem}\label{thm-extend-xu-rigid}
(alternativenss) For each $\bar{K}\in K(\mathcal{M}_{\mathcal{T}})$, up to scaling, $\bar{K}$ is realized by a unique (real or virtual) ball packing in $\mathds{R}^N_{>0}$. In other words, there holds
\begin{equation}
K(\mathcal{M}_{\mathcal{T}})\cap K(\mathds{R}^N_{>0}\setminus\mathcal{M}_{\mathcal{T}})=\emptyset.
\end{equation}
\end{theorem}
\begin{proof}
We need to show any virtual ball packing can not have curvature $\bar{K}$. If not, assume $\bar{r}'$ is a virtual ball packing with curvature $\bar{K}$. Let $\bar{r}$ be the unique (up to scaling) real ball packing with curvature $\bar{K}$. We may well suppose $\|\bar{r}\|_{l^1}=\|\bar{r}'\|_{l^1}=1$. Now we consider the functional $\mathcal{S}_p(r)=\sum_{i=1}^N(K_i-\bar{K}_i)r_i$, which has a natural extension
$$\widetilde{\mathcal{S}}_p(r)=\sum_{i=1}^N(K_i-\bar{K}_i)r_i.$$
Set $\varphi(t)=\widetilde{\mathcal{S}}_p(\bar{r}+t(\bar{r}'-\bar{r}))$,
then we see $\varphi'(0)=\varphi'(1)=0$. Note $\widetilde{\mathcal{S}}_p(r)$ is convex when constricted to the hyperplane $\{r:\|r\|_{l^1}=1\}$, hence $\varphi'(t)$ is monotone increasing. This leads to $\varphi'(t)\equiv0$ for any $t\in[0,1]$. For some small $\epsilon>0$, the functional $\widetilde{\mathcal{S}}_p(r)=\mathcal{S}_p(r)$ is strictly convex when constricted to $B(\bar{r},\epsilon\|\bar{r}'-\bar{r}\|)\cap\{r:\|r\|_{l^1}=1\}$, and for any $t\in[0,\epsilon)$, the ball packing $\bar{r}+t(\bar{r}'-\bar{r})$ is real. Hence $\varphi'(t)$ is strictly monotone increasing for $t\in[0,\epsilon)$, which contradicts with $\varphi'(1)=0$. Thus we get the conclusion above.
\end{proof}

\begin{remark}
\label{Lemma-xu-rigidity}
Similar to the proof of Theorem \ref{thm-extend-xu-rigid}, one may prove Xu's global rigidity \cite{Xu}, i.e. a ball packing is determined by its combinatorial scalar curvature $K$ up to scaling. Consequently, the ball packing with constant curvature (if it exists) is unique up to scaling.
\end{remark}

Theorem \ref{thm-extend-xu-rigid} and its proof have the following interesting corollaries.

\begin{corollary}
There can't be both a real and a virtual packing with constant curvature. Moreover, the set of all constant curvature virtual packings is a convex set in $\mathds{R}^N_{>0}$.
\end{corollary}

The following theorem is a supplement of Theorem \ref{Thm-Q-min-iff-exist-const-curv-metric}.

\begin{theorem}
\label{corollary-Q-2}
Assume there exists a real ball packing $\hat{r}\in \mathcal{M}_{\mathcal{T}}$ with constant curvature.
Then the CRG-functional $\lambda(r)$ has a unique global minimal point in $\mathcal{M}_{\mathcal{T}}$ (up to scaling).
\end{theorem}
\begin{proof}
We restrict our argument on the hyperplane $\{r\in \mathbb{R}^N: ||r||_{l^1}=||\hat{r}||_{l^1}\}$ on which $\mathcal{S}$ and $\lambda$ differ by a constant. Because $\hat{r}$ has constant curvature, by Lemma \ref{lemma-const-curv-equl-cirtical-point}, $\hat{r}$ is a critical point of $\lambda$. In particular, it is a critical point of $\widetilde{\mathcal{S}}$. Theorem \ref{thm-tuta-S-C1-convex} says $\widetilde{\mathcal{S}}$ is global convex on the above hyperplane. Moreover, $\widetilde{\mathcal{S}}$ is local strictly convex near $\hat{r}$, thus we see $\hat{r}$ is the unique global minimum of $\widetilde{\mathcal{S}}$. In particular, it is a global minimum of $\mathcal{S}$.
\end{proof}

\section{The extended flow}
\subsection{Longtime existence of the extended flow}
\label{section-long-exit-extend-flow}
In this subsection, we prove that the solution to the flow (\ref{Def-r-normal-Yamabe flow}) can always be extended to a solution that exists for all time. The basic idea is based on the continuous extension of $K$ to $\widetilde K$. This idea has appeared in the first two authors' former work \cite{GJ1}-\cite{GJ4}.
\begin{theorem}
\label{thm-yang-write}
Consider the normalized combinatorial Yamabe flow (\ref{Def-r-normal-Yamabe flow}). Let $\{r(t)|t\in[0, T)\}$ be the unique maximal solution with $0<T\leq +\infty$. Then we can always extend it to a solution $\{r(t)|t\in[0,+\infty)\}$ when $T<+\infty$. In other words, for any initial real or virtual ball packing $r(0)\in\mathds{R}^N_{>0}$, the solution to the following extended flow
\begin{equation}
\label{Def-Flow-extended}
r_i'(t)=(\tilde \lambda-\widetilde K_i)r_i
  \end{equation}
exists for all time $t\in[0,+\infty)$.
\end{theorem}

\begin{proof}
The proof is similar with Proposition \ref{prop-no-finite-time-I-singula}. Since all $\tilde \lambda-\widetilde K_i$ are continuous functions on $\mathds{R}^N_{>0}$, by Peano's existence theorem in classical ODE theory, the extended flow equation (\ref{Def-Flow-extended}) has at least one solution on some interval $[0,\varepsilon)$. By the definition of $\widetilde{K}_i$, all $|\tilde \lambda-\widetilde{K}_i|$ are uniformly bounded by a constant $c(\mathcal{T})>0$, which depends only on the information of the triangulation. Hence $r_i(0)e^{-c(\mathcal{T})t}\leq r_i(t)\leq r_i(0)e^{c(\mathcal{T})t}$,
which implies that $r_i(t)$ can not go to $0$ or $+\infty$ in finite time. Then by the extension theorem of solutions in ODE theory, the solution exists for all $t\geq 0$.
\end{proof}

\begin{remark}
Set $r_i=2/\sqrt{3}-1$, and all other $r_j=1$ for $j\in V$ and $j\neq i$ in the triangulation $\mathcal{T}$. Recall Example \ref{example-only-c-zero}, we easily get $\partial K_i/\partial r_j=-\infty$ for all vertex $j$ with $j\thicksim i$. This implies that $\widetilde K_i(r)$ is generally not Lipschitz continuous at the boundary point of $\mathcal{M}_{\mathcal{T}}$. So we don't know whether the solution $\{r(t)\}_{t\geq0}$ to the extended flow (\ref{Def-Flow-extended}) is unique.
\end{remark}

Similar to Proposition \ref{prop-V-descending} and Proposition \ref{prop-G-bound} we have the following proposition, the proof of which is omitted.

\begin{proposition}\label{prop-extend-flow-descending}
Along the extended Yamabe flow (\ref{Def-Flow-extended}), $\|r(t)\|_{l^1}$ is invariant. The extended Cooper-Rivin functional $\widetilde{\mathcal{S}}(r)$ is descending and bounded. Moreover, the extended CRG-functional $\tilde{\lambda}(r)$ is descending and uniformly bounded.
\end{proposition}

\subsection{Convergence to constant curvature: general case}
\label{subsection-converg-to-const}
In this section, we prove some convergence results for the extended flow (\ref{Def-Flow-extended}). The following result says that the extended Yamebe flow tends to find real or virtual packings with constant curvature. We omit its proof, since it is similar to Proposition \ref{prop-converg-imply-const-exist}.

\begin{theorem}
\label{thm-extend-flow-converg-imply-exist-const-curv-packing}
If a solution $r(t)$ to the extended flow (\ref{Def-Flow-extended}) converges to some $r_{\infty}\in\mathds{R}^N_{>0}$ as $t\to +\infty$, then $r_{\infty}$ is a constant curvature packing (real or virtual).
\end{theorem}
\begin{remark}
$r_{\infty}$ may be a virtual packing. One can't exclude this case generally.
\end{remark}

\begin{definition}
\label{def-tuta-xi}
Given a triangulated manifold $(M^3, \mathcal{T})$, let $\hat{r}$ be a real ball packing with constant curvature. We introduce an extended combinatorial invariant with respect to the triangulation $\mathcal{T}$ as
\begin{equation}
\tilde{\chi}(\hat{r},\mathcal{T})=\inf\limits_{\gamma\in{\mathbb{S}}^{N-1};\|\gamma\|_{l^1}=0\;}\sup\limits_{0\leq t< \tilde{a}_\gamma}\tilde{\lambda}(\hat{r}+t\gamma),
\end{equation}
where $\tilde{a}_{\gamma}$ is the least upper bound of $t$ such that $\hat{r}+t\gamma\in \mathds{R}_{>0}^N$.
\end{definition}

As is explained in the paragraph before Theorem \ref{Thm-xi-invariant-imply-converg}, the extended invariant $\tilde{\chi}(\hat{r},\mathcal{T})$ is also a pure combinatorial-topological invariant. Obviously,
\begin{equation}
\tilde{\chi}(\hat{r},\mathcal{T})>\chi(\hat{r},\mathcal{T}).
\end{equation}
Similar to Theorem \ref{Thm-xi-invariant-imply-converg}, we have
\begin{theorem}\label{Thm-tuta-xi-invariant-imply-converg}
Assume $\hat{r}$ is a real ball packing with constant curvature. Moreover,
\begin{equation}
\tilde{\lambda}(r(0))\leq\tilde{\chi}(\hat{r},\mathcal{T}).
\end{equation}
Then the solution to the extended normalized flow (\ref{Def-Flow-extended}) exists for all time $t\in[0,+\infty)$ and converges exponentially fast to a real ball packing with constant curvature.
\end{theorem}
\begin{proof}
If $\tilde{\lambda}(r(0))\leq \chi(\hat{r},\mathcal{T})$, then Theorem \ref{Thm-xi-invariant-imply-converg} implies the conclusion directly. We may assume $\tilde{\lambda}(r(0))\geq \chi(\hat{r},\mathcal{T})>Y_{\mathcal{T}}$. Denote $\tilde{\lambda}(t)=\lambda(r(t))$. It's easy to get
\begin{equation}
\label{for-tilde-lambda'}
\tilde{\lambda}'(t)=-\frac{1}{\|r\|_{l^1}}\sum_ir_i(\widetilde{K}_i-\tilde{\lambda})^2\leq0.
\end{equation}
We show $\tilde{\lambda}'(0)<0$. Otherwise $\tilde{\lambda}'(0)=0$ by (\ref{for-tilde-lambda'}), and $\widetilde{K}_i=\tilde{\lambda}$ for all $i\in V$. Hence the real or virtual packing $r(0)$ has constant curvature. By Theorem \ref{thm-extend-xu-rigid} the alternativeness, $r(0)$ must be a real ball packing. Hence by Theorem \ref{corollary-Q-2}, we get $\tilde{\lambda}(r(0))=Y_{\mathcal{T}}$, which is a contradiction. Hence $\tilde{\lambda}'(0)<0$. Therefore, $\tilde{\lambda}(r(t))$ is strictly descending along (\ref{Def-Flow-extended}) for at least a small time interval $t\in[0,\epsilon)$. Thus $r(t)$ will never touche the boundary of $\mathds{R}_{>0}^N$ along (\ref{Def-Flow-extended}). By classical ODE theory, the solution $r(t)$ exists for all time $t\in[0,+\infty)$. Moreover, $\{r(t)\}_{t\geq0}$ is compactly supported in $\mathds{R}_{>0}^N$. Consider the functional $\widetilde{\mathcal{S}}_Y=\sum_i(\widetilde{K}_i-Y_{\mathcal{T}})r_i$. By the extended Schl\"{a}ffli formula in Corollary \ref{coro-extend-schlafi-formula}, we get $d\widetilde{\mathcal{S}}_Y=\sum_i(\widetilde{K}_i-Y_{\mathcal{T}})dr_i$. As a consequence, along (\ref{Def-Flow-extended}) we have
$$\widetilde{\mathcal{S}}'_Y(t)=-\|r\|^{-1}_{l^1}\sum_ir_i(\widetilde{K}_i-\tilde{\lambda})^2\leq0.$$
Hence then $\widetilde{\mathcal{S}}_Y(t)$ is descending. Because $\widetilde{\mathcal{S}}_Y\geq0$, so $\widetilde{\mathcal{S}}_Y(+\infty)$ exists. Hence there is a time sequence $t_n\uparrow+\infty$, such that
$$\widetilde{\mathcal{S}}'_Y(t_n)=\widetilde{\mathcal{S}}_Y(n+1)-\widetilde{\mathcal{S}}_Y(n)\rightarrow0.$$
Note $\tilde{\lambda}(+\infty)$ exists by (\ref{for-tilde-lambda'}) and Proposition \ref{prop-extend-flow-descending}. It follows that
$$r_i(t_n)(\widetilde{K}_i(t_n)-\tilde{\lambda}(+\infty))\rightarrow0$$
at each vertex $i\in V$. Because $\{r(t)\}_{t\geq0}$ is compactly supported in $\mathds{R}_{>0}^N$, we may choose a subsequence of $\{t_n\}_{n\geq1}$, which is still denoted as $\{t_n\}_{n\geq1}$ itself, so that $r(t_n)\rightarrow r^*\in\mathds{R}_{>0}^N$. Because the extended curvature $\widetilde{K}_i$ is continuous, so
$\widetilde{K}_i(t_n)\rightarrow\widetilde{K}_i(r^*)$. This implies that $r^*$ is a real or virtual ball packing with constant curvature.  By Theorem \ref{thm-extend-xu-rigid} the alternativeness, $r^*$ must be a real ball packing. Because $r^*$ is asymptotically stable (see Lemma \ref{Thm-3d-isolat-const-alpha-metric}), and $r(t)$ goes to $r^*$ along the time sequence $\{t_n\}$, then $r(t)$ converges to $r^*$ as $t$ goes to $+\infty$. By Lemma \ref{Lemma-ODE-asymptotic-stable}, the convergence rate of $r(t)$ is exponential.
\end{proof}



\begin{definition}
The extended combinatorial Yamabe invariant $\widetilde{Y}_{\mathcal{T}}$ (with respect to a triangulation $\mathcal{T}$) is
\begin{equation}
\widetilde{Y}_{\mathcal{T}}=\inf_{r\in \mathds{R}_{>0}^N} \tilde{\lambda}(r)=\inf_{r\in \mathds{R}_{>0}^N}\frac{\sum_{i=1}^N \tilde K_i r_i}{\sum_{i=1}^N r_i}.
\end{equation}
\end{definition}

From the definition of $\widetilde{Y}_{\mathcal{T}}$, we see $\widetilde{Y}_{\mathcal{T}}\leq Y_{\mathcal{T}}$. Moreover, by Corollary \ref{corollary-Q-2} we have
\begin{corollary}\label{Thm-Y-equal-tuta-Y}
Assume there exists a real packing $\hat{r}\in \mathcal{M}_{\mathcal{T}}$ with constant curvature. Then
\begin{equation}
\widetilde{Y}_{\mathcal{T}}=Y_{\mathcal{T}}.
\end{equation}
\end{corollary}

\begin{conjecture}
Assume $\widetilde{Y}_{\mathcal{T}}=Y_{\mathcal{T}}$, then there exists a ball packing (real or virtual) with constant curvature.
\end{conjecture}

We say the extended combinatorial Yamabe invariant $\widetilde{Y}_{\mathcal{T}}$ is \emph{attainable} if the extended CRG-functional $\tilde{\lambda}(r)$ has a global minimum in $\mathds{R}^N_{>0}$. Similar to Theorem \ref{Thm-Q-min-iff-exist-const-curv-metric} and Theorem
\ref{corollary-Q-2}, we have
\begin{theorem}\label{Thm-tuta-yamabe-invarint}
Given a triangulated manifold $(M^3, \mathcal{T})$, the following four descriptions are mutually equivalent.
\begin{enumerate}
  \item There exists a real or virtual ball packing $\hat{r}$ with constant curvature.
  \item The extended CRG-functional $\tilde{\lambda}(r)$ has a local minimum in $\mathds{R}^N_{>0}$.
  \item The extended CRG-functional $\tilde{\lambda}(r)$ has a global minimum in $\mathds{R}^N_{>0}$.
  \item The extended Yamabe invariant $\widetilde{Y}_{\mathcal{T}}$ is attainable by a real or virtual ball packing.
\end{enumerate}
Moreover, if $\widetilde{Y}_{\mathcal{T}}$ is attained by a virtual packing, the set of virtual packings that realized $\widetilde{Y}_{\mathcal{T}}$ equals to the set of constant curvature virtual packings, which form a convex set in $\mathds{R}^N_{>0}$.
\end{theorem}

\subsection{Convergence with regular triangulation}
\label{section-regular-converge}
The main purpose of this section is to prove the following theorem:
\begin{theorem}\label{Thm-regular-triangu-converge}
Assume the triangulation $\mathcal{T}$ is regular. Then the solution $\{r(t)\}_{t\geq0}$ to the extended Yamabe flow (\ref{Def-Flow-extended}) converges exponentially fast to the unique real packing with constant curvature as $t$ goes to $+\infty$.
\end{theorem}

Before giving the proof of the above theorem, we need to study some deep relations between $r$ and $\tilde{\alpha}$. We also need to compare a conformal tetrahedron to a regular tetrahedron. By definition, a regular tetrahedron is one with all four radii equal. Hence all four solid angles are equal in a regular tetrahedron. We denote this angle by $\bar{\alpha}$, i.e.
\begin{equation}
\bar{\alpha}=3\cos^{-1}\frac{1}{3}-\pi.
\end{equation}
Let $\tau=\{1234\}$ be a conformal tetrahedron (real or virtual) patterned by four mutually externally tangent balls with radii $r_1$, $r_2$, $r_3$ and $r_4$. For each vertex $i\in\{1,2,3,4\}$, let $\alpha_i$ be the solid angle at $i$. Recall $\tilde{\alpha}_i$ is the continuous extension of $\alpha_i$.

\begin{lemma}\label{Lemma-Glicken-big-r-small-angle}
(Glickenstein, Lemma 7 \cite{G2})  $\alpha_i\geq\alpha_j$ if and only if $r_i\leq r_j$.
\end{lemma}
It's easy to show that Glickenstein's Lemma \ref{Lemma-Glicken-big-r-small-angle} also holds true for virtual tetrahedrons, that is, $\tilde{\alpha}_i\geq\tilde{\alpha}_j$ if and only if $r_i\leq r_j$. The following two lemmas are very important. They establish two comparison principles for the extended solid angles between a general tetrahedron (real or virtual) and a regular one.

\begin{lemma}\label{Lemma-compare-regular-1}
(first comparison principle) If $r_j$ is maximal, then $\tilde{\alpha}_j\leq\bar{\alpha}$.
\end{lemma}
\begin{proof}
Let $r_j=\max\{r_1, r_2, r_3, r_4\}$ be maximal. If the tetrahedron $\tau$ is virtual, then either $\tilde{\alpha}_j=0$ or $\tilde{\alpha}_j=2\pi$. By the definition of the solid angle $\tilde{\alpha}$ (see Section \ref{section-extend-solid-angle}), $\tilde{\alpha}_j=2\pi$ implies that $\tau\in D_j$. By Lemma \ref{lemma-Di-imply-ri-small}, $r_j$ is strictly minimal, which is a contradiction. Hence we get $\tilde{\alpha}_j=0<\bar{\alpha}$. If the conformal tetrahedron $\tau$ is real, by Lemma \ref{Lemma-Glicken-big-r-small-angle}, $\alpha_j$ is the minimum of $\{\alpha_1, \alpha_2, \alpha_3, \alpha_4\}$. Denote $r=(r_1,r_2,r_3,r_4)$. Consider the functional
$$\varphi(r)=\sum_{i=1}^4(\alpha_i-\bar{\alpha})r_i,\;r\in\Omega_{1234}.$$
By the Schl\"{a}ffli formula, we get $d\varphi=\sum_{i=1}^4(\alpha_i-\bar{\alpha})dr_i$. Taking the differential,  we have
$$\text{Hess}\varphi=\frac{\partial(\alpha_1,\alpha_2,\alpha_3,\alpha_4)}{\partial(r_1,r_2,r_3,r_4)}.$$
By Lemma \ref{Lemma-Lambda-semi-positive}, Hess$\varphi$ is negative semi-definite with rank $3$ and the kernel $\{tr:t\in\mathds{R}\}$. Hence $\varphi$ is strictly local concave when restricted to the hyperplane $\{r:\sum_{i=1}^4r_i=1\}$. Similarly (see Section \ref{section-extend-CR-funct} and Theorem \ref{thm-tuta-S-C1-convex}), the extended functional
$$\tilde{\varphi}(r)=\sum_{i=1}^4(\tilde{\alpha}_i-\bar{\alpha})r_i,\;r\in\mathds{R}_{>0}^4$$
is $C^1$-smooth and is concave on $\mathds{R}_{>0}^4$. Because it equals to $\varphi$ on $\Omega_{1234}$, it is $C^{\infty}$-smooth on $\Omega_{1234}$ and is strictly concave on
$\Omega_{1234}\cap\{r:\sum_{i=1}^4r_i=1\}$. Note $d\tilde{\varphi}=\sum_{i=1}^4(\tilde{\alpha_i}-\bar{\alpha})dr_i$, we see $\nabla\tilde{\varphi}=\tilde{\alpha}-\bar{\alpha}$, implying that the regular radius $\bar{r}=(1,1,1,1)$ is the unique critical point of $\tilde{\varphi}$. It follows that $\tilde{\varphi}$ has a unique maximal point at $\bar{r}$. Thus
$\varphi(r)\leq\varphi(\bar{r})=0$. By Glickenstein's Lemma \ref{Lemma-Glicken-big-r-small-angle}, we see $\alpha_j=\min\{\alpha_1,\alpha_2,\alpha_3,\alpha_4\}$. The conclusion follows from
$$\min\{\alpha_1,\alpha_2,\alpha_3,\alpha_4\}\leq\frac{\sum_{i=1}^4\alpha_ir_i}{\sum_{i=1}^4 r_i}\leq\bar{\alpha}.$$
\end{proof}

\begin{lemma}\label{Lemma-compare-regular-2}
(second comparison principle) If $r_i$ is minimal, then $\tilde{\alpha}_i\geq\bar{\alpha}$.
\end{lemma}
\begin{proof}
Let $r_i=\min\{r_1, r_2, r_3, r_4\}$ be minimal. Let $j$, $k$ and $l$ be the other three vertices in $\{1,2,3,4\}$ which is different with $i$. We first prove the following two facts:
\begin{enumerate}
  \item If the conformal tetrahedron $\tau$ is virtual, then $\tilde{\alpha}_i=2\pi$;
  \item If $\tau$ is real, then $\partial\alpha_i/\partial r_j$, $\partial\alpha_i/\partial r_k$ and $\partial\alpha_i/\partial r_l$ are positive, while $\partial\alpha_i/\partial r_i$ is negative.
\end{enumerate}

To get the above fact 1, we assume $\tau$ is virtual, then either $\tilde{\alpha}_i=0$ or $\tilde{\alpha}_i=2\pi$. However, $\tilde{\alpha}_i=0$ is impossible. In fact, if this happens, then $\tilde{\alpha}_p=2\pi$ for some $p\in\{1,2,3,4\}$ with $p\neq i$. By the definition of the extended solid angles $\tilde{\alpha}$ (see Section \ref{section-extend-solid-angle}), we know $\tau\in D_p$. By Lemma \ref{lemma-Di-imply-ri-small}, $r_p$ is strictly minimal. This contradicts with the assumption that $r_i$ is minimal. The only possible case left is $\tilde{\alpha}_i=2\pi$, which implies the conclusion.

To get the above fact 2, we assume $\tau$ is non-degenerate, or say real. Glickenstein (see formula (7) in \cite{G1}) once calculated
$$\frac{\partial\alpha_i}{\partial r_j}=\frac{4r_ir_jr_k^2r^2_l}{3P_{ijk}P_{ijl}V_{ijkl}}
\left(\frac{1}{r_i}\left(\frac{1}{r_j}+\frac{1}{r_k}+\frac{1}{r_l}\right)+\frac{1}{r_j}\left(\frac{1}{r_i}+\frac{1}{r_k}+\frac{1}{r_l}\right)-
\left(\frac{1}{r_k}-\frac{1}{r_l}\right)^2\right),$$
where $P_{ijk}=2(r_i+r_j+r_k)$ is the perimeter of the triangle $\{ijk\}$, $V_{ijkl}$ is the volume of the conformal tetrahedron $\tau$. Since $r_i$ is minimal, we get
\begin{align*}
&\frac{1}{r_i}\left(\frac{1}{r_j}+\frac{1}{r_k}+\frac{1}{r_l}\right)+\frac{1}{r_j}\left(\frac{1}{r_i}+\frac{1}{r_k}+\frac{1}{r_l}\right)-
\left(\frac{1}{r_k}-\frac{1}{r_l}\right)^2\\[8pt]
&>\frac{1}{r_ir_j}-\frac{1}{r^2_j}+\frac{1}{r_ir_l}-\frac{1}{r^2_l}=\frac{r_j-r_i}{r_ir^2_j}+\frac{r_l-r_i}{r_ir^2_l}\geq0.
\end{align*}
Hence $\partial\alpha_i/\partial r_j>0$. Similarly, we have $\partial\alpha_i/\partial r_k>0$ and $\partial\alpha_i/\partial r_l>0$. From
$$\frac{\partial\alpha_i}{\partial r_i}r_i+\frac{\partial\alpha_i}{\partial r_j}r_j+
\frac{\partial\alpha_i}{\partial r_k}r_k+\frac{\partial\alpha_i}{\partial r_l}r_l=0,$$
we further get $\partial\alpha_i/\partial r_i<0$. Thus we get the above two facts.

Now we come back to the initial setting, where $r_i=\min\{r_1, r_2, r_3, r_4\}$ is minimal. Since the final conclusion is symmetric with respect to $r_j$, $r_k$ and $r_l$, we may suppose
$$r_i\leq r_j\leq r_k\leq r_l.$$
If the initial conformal tetrahedron $\tau$ is virtual, then we get the conclusion already. If the initial conformal tetrahedron $\tau$ is real, then $\alpha_i<2\pi$ by definition. We use a continuous method to get the final conclusion. We approach it by three steps:

Step 1. Let $r_i$ increase to $r_j$. From the above fact 2 we get $\partial\alpha_i/\partial r_i<0$. It follows that $\alpha_i$ is descending. Moreover, along this procedure there maintains $r_i\leq r_j\leq r_k\leq r_l$, hence the degeneration will never happen along this procedure (otherwise $\tilde{\alpha}_i=2\pi$ by the above fact 1, contradicting with that $\alpha_i<2\pi$ is descending). By this time, we get
$$r_i=r_j\leq r_k\leq r_l.$$

Step 2. Let $r_k$ decrease to $r_j=r_i$. From the above fact 2 we get $\partial\alpha_i/\partial r_k>0$. It follows that $\alpha_i$ is descending. Moreover, along this procedure there maintains $r_i=r_j\leq r_k\leq r_l$, hence the degeneration will never happen (otherwise $\tilde{\alpha}_i=2\pi$ by the above fact 1, contradicting with pthat $\alpha_i<2\pi$ is descending). By this time, we get
$$r_i=r_j=r_k\leq r_l.$$

Step 3. Let $r_l$ decreases to $r_k=r_j=r_i$. Similar to Step 2, from $\partial\alpha_i/\partial r_l>0$ we see that $\alpha_i$ is descending. Moreover, along this procedure there maintains $r_i=r_j=r_k\leq r_l$, hence there is no degeneration happen (otherwise $\tilde{\alpha}_i=2\pi$, contradicting with that $\alpha_i<2\pi$ is descending). By the time, we finally get
$$r_i=r_j=r_k=r_l.$$

Note along these procedure, $\alpha_i$ is always descending and at last tends to $\bar{\alpha}$. Hence we havep $\alpha_i\geq\bar{\alpha}$, which is the conclusion.
\end{proof}

Now it's time to prove Theorem \ref{Thm-regular-triangu-converge}. The following proof can be viewed as to derive a Harnack-type inequality.\\

\noindent\emph{Proof of Theorem \ref{Thm-regular-triangu-converge}.}
Assume at some time $t$, $r_i(t)$ is minimal while $r_j(t)$ is maximal. By the two comparison principles Lemma \ref{Lemma-compare-regular-1} and Lemma \ref{Lemma-compare-regular-2}, we see for each tetrahedron with vertex $i$ that $\tilde{\alpha}_i\geq\bar{\alpha}$ and each tetrahedron with vertex $j$ that $\bar \alpha\geq\tilde{\alpha}_j$.
Compare the extended flow $d\ln r_i/dt=\tilde \lambda-\widetilde{K}_i$ at $i$ and $d\ln r_j/dt=\widetilde \lambda-\widetilde{K}_j$ at $j$, their difference is
\begin{align*}
\frac{d}{dt}\left(\min\limits_{p,q\in V}\Big\{\frac{r_{p}(t)}{r_{q}(t)}\Big\}\right)&=\frac{d}{dt}\left(\frac{r_i(t)}{r_j(t)}\right)\\[5pt]
&=\frac{r_i(t)}{r_j(t)}\big(\widetilde{K}_j-\widetilde{K}_i\big)\\[5pt]
&=\frac{r_i(t)}{r_j(t)}\left(\sum \tilde{\alpha}_i-\sum \tilde{\alpha}_j\right)\geq0.
\end{align*}
It follows that $\min\limits_{p,q\in V}\Big\{\cfrac{r_{p}(t)}{r_{q}(t)}\Big\}$ is non-descending along the extended flow (\ref{Def-Flow-extended}). Hence there is a constant $c>1$ so that
$$c^{-1}r_p(t)\leq r_q(t)\leq c r_p(t).$$
This implies that the solution $\{r(t)\}_{t\geq0}$ lies in a compact subset of $\mathds{R}^N_{>0}$. Similar to the proof of Theorem \ref{Thm-nosingular-imply-converg} and Theorem \ref{Thm-tuta-xi-invariant-imply-converg},
we can find a particular real or virtual ball packing $\hat{r}\in\mathds{R}^N_{>0}$ that has constant curvature. Note all $r_i=1$ provides a real ball packing with constant curvature, and the alternativeness Theorem \ref{thm-extend-xu-rigid} says that a constant curvature real packing and a constant curvature virtual packing can't exist simultaneously, we know $\hat{r}$ is a real packing. Thus we get the conclusion.\qed

\begin{remark}
We can't exclude the possibility that $r(t)$ runs outside $\mathcal{M}_{\mathcal{T}}$ on the way to $\hat{r}$. Consider the extreme case that the initial packing $r(0)$ is virtual. However, any solution $r(t)$ will eventually runs into $\mathcal{M}_{\mathcal{T}}$ and converges to $\hat{r}$ exponentially fast.
\end{remark}

We want to prove the convergence of (\ref{Def-Flow-extended}) under the only assumption that there exists a real packing with constant curvature. Indeed, we can do so for combinatorial Yamabe (or Ricci, Calabi) flows on surfaces, see \cite{Ge}-\cite{GJ4}\cite{GX2}-\cite{ge-xu}. In three dimension, the trouble comes from that $\widetilde{\mathcal{S}}(r)$ is not proper on the hyperplane $\{r:\sum_ir_i=1\}$. We guess there is a suitable continuous extension of $K_i$, so that $\widetilde{\mathcal{S}}(r)$ is proper, and hence the corresponding extended flow converges.

\subsection{A conjecture for degree difference $\leq10$ triangulations}
Inspired by the two comparison principles Lemma \ref{Lemma-compare-regular-1} and Lemma \ref{Lemma-compare-regular-2} and the proof of Theorem \ref{Thm-regular-triangu-converge}, we pose the following
\begin{conjecture}
\label{conjec-differ-small-11}
Let $d_i$ be the vertex degree at $i$. Assume $|d_i-d_j|\leq 10$ for each $i,j\in V$. Then there exists a real or virtual ball packing with constant curvature.
\end{conjecture}
In the abvoe conjecture, the basic assumption is \emph{combinatorial}, while the conclusion is \emph{geometric}. Thus it builds a connection between combinatoric and geometry. It says that the combinatoric effects geometry. It is based on the following intuitive but not so rigorous observation:
If the solution $r(t)$ lies in a compact subset of $\mathds{R}^N_{>0}$, then $r(t)$ converges to some real or virtual packing $\hat{r}$ with constant curvature. This had been used in the proof of Theorem \ref{Thm-regular-triangu-converge}. If $r(t)$ does not lie in any compact subset of $\mathds{R}^N_{>0}$, then it touches the boundary of $\mathds{R}^N_{>0}$. Because $\sum_ir_i(t)$ is invariant, so all $r_i(t)$ are bounded above. Thus at least one coordinates of $r(t)$, for example $r_i(t)$, goes to $0$ as the time goes to infinity. We may assume $r_i(t)$ is minimal, while $r_p(t)$ is maximal. It seems that there is at least one tetrahedron $\{ijkl\}$ going to degenerate. It's easy to see $\sum\tilde{\alpha}_i\geq 2\pi+(d_i-1)\bar{\alpha}$. Hence by the assumption that all degree differences are no more than $10$, there holds
\begin{equation*}
\sum\tilde{\alpha}_i-\sum\tilde{\alpha}_p\geq 2\pi+(d_i-1)\bar{\alpha}-d_p\bar{\alpha}>2\pi-11\bar{\alpha}>0.
\end{equation*}
Therefore
$$\frac{d}{dt} \left(\frac{r_{i}(t)}{r_{p}(t)}\right)=\left(\frac{r_i(t)}{r_p(t)}\right)\big(\sum\tilde{\alpha}_i-\sum\tilde{\alpha}_p\big)>0.$$
This shows that $r_i(t)/r_p(t)$ has no tendency of descending, which contradicts with $r_i(t)$ goes to $0$. From above analysis, we see the main difficulty is how to show at the minimum $r_i(t)$, there is a tetrahedron $\{ijkl\}$ going to degenerate. This may be overcome by combinatorial techniques. However, we can adjust the above explanation to show
\begin{theorem}
If each vertex degree is no more than $11$, there exists a real or virtual ball packing with constant curvature.
\end{theorem}
\begin{proof}
We sketch the proof by contradiction. If there is a tetrahedron $\{ijkl\}$ tends to degenerate at infinity, then there is at least a solid angle, for example $\alpha_{ijkl}$ tends to $2\pi$. Hence $r_i$ is the strictly minimum of $\{r_i,r_j,r_k,r_l\}$. $r_i$ tends to $0$ at infinity, otherwise $r(t)$ will lie in a compact set of $\mathds{R}^N_{>0}$. Let $r_p(t)$ be maximal, then
$$\frac{d}{dt} \left(\frac{r_{i}(t)}{r_{p}(t)}\right)=
\left(\frac{r_i(t)}{r_p(t)}\right)\Big(\sum\tilde{\alpha}_i-\sum\tilde{\alpha}_p\Big)\geq \left(\frac{r_i(t)}{r_p(t)}\right)(2\pi-d_p\bar{\alpha})>0.$$
This leads to a contradiction.
\end{proof}

\subsection{A prescribed curvature problem}
For any $\bar{K}=(\bar{K}_1, \cdots, \bar{K}_N)$, we want to know if it can be realized as the combinatorial scalar curvature of some real or virtual ball packing $\bar{r}$. In other words, is there a packing $\bar{r}$ so that the corresponding curvature $K(\bar{r})=\bar{K}$. Similarly, we consider the following\\[8pt]
\noindent
\textbf{Prescribed Curvature Problem:} \emph{Is there a real ball packing with the prescribed combinatorial scalar curvature $\bar{K}$ in the combinatorial conformal class $\mathcal{M}_{\mathcal{T}}$? How to find it?}\\[8pt]
By Xu's global rigidity (see Remark \ref{Lemma-xu-rigidity}), a real ball packing $\bar{r}$ that realized $\bar{K}$ is unique up to scaling. To study the above ``Prescribed Curvature Problem", we introduce:
\begin{definition}
Any $\bar{K}\in\mathds{R}^N_{>0}$ is called a prescribed curvature. Given any prescribed curvature $\bar{K}$, the prescribed Cooper-Rivin functional is defined as $$\mathcal{S}_p(r)=\sum_{i=1}^N(K_i-\bar{K}_i)r_i.$$
The prescribed CRG-functional is defined as $\lambda_p(r)=\mathcal{S}_p(r)/\|r\|_{l^1}$.
\end{definition}
We summarize some results related to the ``Prescribed Curvature Problem" as follows. We omit their proofs since they are similar to the results in previous sections.

\begin{theorem}\label{Thm-Q-bar-curv}
Given a triangulated manifold $(M^3, \mathcal{T})$. The following three properties are mutually equivalent.
\begin{enumerate}
  \item The ``Prescribed Curvature Problem" is solvable. That is, there exists a real packing $\bar{r}$ that realizes $\bar{K}$.
  \item The prescribed CRG-functional $\lambda_p(r)$ has a local minimum in $\mathcal{M}_{\mathcal{T}}$.
  \item The prescribed CRG-functional $\lambda_p(r)$ has a global minimum in $\mathcal{M}_{\mathcal{T}}$.
\end{enumerate}
\end{theorem}

\begin{proposition}
If the prescribed combinatorial Yamabe flow
\begin{equation}\label{Def-Flow-modified}
\frac{dr_i}{dt}=(\bar{K}_i-K_i)r_i
\end{equation}
converges, then $r(+\infty)$, the solution to (\ref{Def-Flow-modified}) at infinity, solves the ``Prescribed Curvature Problem". Conversely, if the ``Prescribed Curvature Problem" is solvable by a real packing $\bar{r}$, and if the initial real packing $r(0)$ deviates from $\bar{r}$ not so much. then the solution $r(t)$ to (\ref{Def-Flow-modified}) exists for all time and converges exponentially fast to $\bar{r}$.
\end{proposition}

\begin{definition}
Let $\bar{r}\in \mathcal{M}_{\mathcal{T}}$ be a real ball packing. Define a prescribed combinatorial-topological invariant (with respect to $\bar{r}$ and $\mathcal{T}$) as
\begin{equation}
\chi(\bar{r},\mathcal{T})=\inf\limits_{\gamma\in{\mathbb{S}}^{N-1};\|\gamma\|_{l^1}=1\;}\sup\limits_{0\leq t< a_{\bar{r},\gamma}}\lambda_p(\bar{r}+t\gamma),
\end{equation}
where $a_{\bar{r},\gamma}$ is the least upper bound of $t$ such that $\bar{r}+t\gamma\in \mathcal{M}_{\mathcal{T}}$.
\end{definition}

\begin{theorem}\label{Thm-prescrib-xi-invariant-imply-converg}
Let $\bar{r}$ be a real ball packing with curvature $\bar{K}=K(\bar{r})$. Consider the prescribed flow (\ref{Def-Flow-modified}). If the initial real packing $r(0)$ satisfies
\begin{equation}
\lambda_p(r(0))\leq\chi(\bar{r},\mathcal{T}).
\end{equation}
Then the solution to (\ref{Def-Flow-modified}) exists for all time $t\geq 0$ and converges exponentially fast to $\bar{r}$.
\end{theorem}

\begin{theorem}
Given any initial ball packing (real or virtual) $r(0)\in\mathds{R}^N_{>0}$, the following extended prescribed combinatorial Yamabe flow
\begin{equation}
  \label{Def-Flow-modified-extended}
   r_i'(t)=(\bar K_i-\tilde K_i)r_i
\end{equation}
has a solution $r(t)$ with $t\in[0,+\infty)$. In other words, the solution to (\ref{Def-Flow-modified}) can always be extended to a solution that exists for all time $t\geq 0$.
\end{theorem}

One can also extend the prescribed CRG-functional $\lambda_p(r)$, $r\in\mathcal{M}_{\mathcal{T}}$ to $\tilde{\lambda}_p(r)$, $r\in\mathds{R}^N_{>0}$, and introduce a combinatorial-topological invariant (with respect to a real ball packing $\bar{r}$ and the triangulation $\mathcal{T}$)
\begin{equation}
\tilde{\chi}(\bar{r},\mathcal{T})=\inf\limits_{\gamma\in{\mathbb{S}}^{N-1};\|\gamma\|_{l^1}=1\;}\sup\limits_{0\leq t< \tilde{a}_{\bar{r},\gamma}}\tilde{\lambda}_p(\bar{r}+t\gamma),
\end{equation}
where $\tilde{a}_{\bar{r},\gamma}$ is the least upper bound of $t$ such that $\bar{r}+t\gamma\in \mathds{R}_{>0}^N$. Similar to Theorem \ref{Thm-tuta-xi-invariant-imply-converg}, we have
\begin{theorem}
\label{Thm-extend-prescrib-xi-invariant-imply-converg}
Let $\bar{r}$ be a real ball packing with curvature $\bar{K}=K(\bar{r})$. If the initial ball packing $r(0)$ (real or virtual) satisfies
\begin{equation}
\tilde{\lambda}_p(r(0))\leq\tilde{\chi}(\bar{r},\mathcal{T}),
\end{equation}
then the solution to (\ref{Def-Flow-modified-extended}) exists for all time $t\geq 0$ and converges exponentially fast to $\bar{r}$.
\end{theorem}

\section{Appendix}
\subsection{The Schl\"{a}ffli formula}
\label{appen-schlafi}
Given an Euclidean tetrahedron $\tau$ with four vertices $1,2,3,4$. For each edge $\{ij\}$, denote $l_{ij}$ as the edge length of $\{ij\}$, and denote $\beta_{ij}$ as the dihedral angle at the edge $\{ij\}$. The classical schl\"{a}ffli formula reads as $\sum_{i\thicksim j}l_{ij}d\beta_{ij}=0$,
where the sum is taken over all six edges of $\tau$. If $\tau$ is configured by four mutually externally tangent ball packings $r_1, r_2, r_3$ and $r_4$, then on one hand,
\begin{align*}
d\Big(\sum_{i\thicksim j}l_{ij}\beta_{ij}\Big)=\sum_{i\thicksim j}\beta_{ij}dl_{ij}&=\sum_{i\thicksim j}\beta_{ij}\big(dr_i+dr_j\big)\\
&=\sum_i\Big(\sum_{j:j\thicksim i}\beta_{ij}\Big)dr_i=\sum_i\big(\alpha_i+\pi\big)dr_i.
\end{align*}
On the other hand,
$$\sum\limits_{i\thicksim j}l_{ij}\beta_{ij}=\sum\limits_{i\thicksim j}\beta_{ij}(r_i+r_j)=\sum_{i}\Big(\sum\limits_{j:j\thicksim i}\beta_{ij}\Big)r_i=\sum_{i}\big(\alpha_i+\pi\big)r_i.$$
Hence we obtain $d\Big(\sum_i\alpha_ir_i\Big)=\sum_{i}\alpha_idr_i$. Consider a triangulated manifold $(M^3,\mathcal{T})$ with ball packings $r$, we can further get $d\Big(\sum_iK_ir_i\Big)=\sum_{i}K_idr_i$.

\subsection{The proof of Lemma \ref{Lemma-Lambda-positive}}
\label{appendix-2}
\begin{proof} Denote $\mathscr{U}=\{u\in\mathds{R}^N|\sum_iu_i=0\}$.
Set $\alpha=(1,\cdots,1)^T/\sqrt{N}$. Choose $A\in O(N)$, such that $A\alpha=(0,\cdots,0,1)^T$, meanwhile, $A$ transforms $\mathscr{U}$ to $\{\zeta\in\mathds{R}^N|\zeta_N=0\}$. It's easy to see $A$ transforms $\{r\in\mathds{R}^N|\sum_ir_i=1\}$ to $\{\zeta\in\mathds{R}^N|\zeta_N=1/\sqrt{N}\}$. Define $g(\zeta_1,\cdots,\zeta_{N-1})\triangleq S(A^T(\zeta_1,\cdots,\zeta_{N-1},1/\sqrt{N})^T)$, we can finish the proof by showing that $Hess_{\zeta}(g)$ is positive definite. Because $A\alpha=(0,\cdots,0,1)^T$, we get $\alpha^T=(0,\cdots,0,1)A$, which implies that we can partition $A$ into two blocks with $A^T=\left[B^{N\times(N-1)}, \alpha\right]$. By direct calculation, we get $\nabla_{\zeta}g=B^TK$ and $Hess_{\zeta}(g)=B^T\Lambda B$. Next we prove $B^T\Lambda B$ positive definite. Assuming $x^TB^T\Lambda Bx=0$, where $x$ is a nonzero $(N-1)\times1$ vector. From Lemma \ref{Lemma-Lambda-semi-positive}, there exists a $c\neq0$ such that $Bx=cr$. On the other hand, from
\begin{equation*}
\begin{bmatrix}
I_{N-1}& 0\,\,\\
0 & 1\,\,
\end{bmatrix}=AA^T=
\begin{bmatrix}
B^T \\ \alpha^T
\end{bmatrix}
\big[B,\alpha\big]=
\begin{bmatrix}
B^TB & B^T\alpha \\
\alpha^TB & \alpha^T\alpha
\end{bmatrix}
\end{equation*}
we know $\alpha^TB=0$. Then $0=\alpha^TBx=c\alpha^Tr=c/\sqrt{N}$, which is a contradiction. Hence $x^TB^T\Lambda Bx=0$ has no nonzero solution. Note that $B^T\Lambda B$ is positive semi-definite due to Lemma \ref{Lemma-Lambda-semi-positive}, thus $B^T\Lambda B$ is positive definite.
\end{proof}

\noindent
\textbf{Acknowledgements:} The first two authors want to thank Bennett Chow for reminding us that (\ref{Def-norm-Yamabe-Flow}) was first introduced in Glickenstein's thesis \cite{G0}. The first author would like to thank Professor Feng Luo, Jie Qing, Xiaochun Rong for many helpful conversations. The first author is partially supported by the NSFC of China (No.11501027). The second author was supported by the Fundamental Research Funds for the Central Universities (No. 2017QNA3001) and NSFC (No. 11701507).

\noindent Huabin Ge, hbge@bjtu.edu.cn\\[2pt]
\emph{Department of Mathematics, Beijing Jiaotong University, Beijing 100044, P.R. China}\\[2pt]

\noindent Wenshuai Jiang, wsjiang@zju.edu.cn\\[2pt]
\emph{School of Mathematical Sciences, Zhejiang University, Zheda Road 38, Hangzhou, 310027, P.R. China}\\[2pt]

\noindent Liangming Shen, lmshen@math.ubc.ca\\[2pt]
\emph{Department of Mathematics, The University of British Columbia, Room 121, 1984 Mathematics Road, Vancouver, B.C., Canada V6T 1Z2}
\end{document}